\documentclass[11pt,leqno]{amsart}

\usepackage{cite}
\usepackage{amssymb}
\usepackage{amsmath}
\usepackage{amsthm}
\usepackage{amsfonts}
\usepackage{bbm}
\usepackage{enumerate} 
\usepackage{graphicx}
\usepackage{xcolor}
\usepackage{hyperref}
\hypersetup{
	colorlinks=true,
	linkcolor=blue,    
	citecolor=green,   
	urlcolor=red       
}
\usepackage[toc,page]{appendix}
\usepackage[normalem]{ulem}
\usepackage{comment}
\usepackage{enumitem}

\newcommand{\Vol}{\mathrm{Vol}}
\newcommand{\R}{\mathbb{R}}
\renewcommand{\P}{\mathbb{P}}
\newcommand{\supp}{\mathrm{supp}}
\newcommand{\Per}{\mathrm{Per}}
\newcommand{\E}{\mathbb{E}}
\renewcommand{\ge}{\geqslant}
\renewcommand{\le}{\leqslant}
\renewcommand{\leq}{\leqslant}

\newcommand{\dist}{\mathrm{dist}}

\newcommand{\BV}{\mathrm{BV}}
\renewcommand{\epsilon}{\varepsilon}

\makeatother
\numberwithin{equation}{section}

\newtheorem{theorem}{Theorem}[section]
\newtheorem{assumption}{Assumption}[section]
\newtheorem{proposition}[theorem]{Proposition}

\newtheorem{lemma}[theorem]{Lemma}

\theoremstyle{definition}

\theoremstyle{remark}
\newtheorem{remark}[theorem]{Remark}
\newtheorem{notation}{Notation}[section]

\makeatletter
\@namedef{subjclassname@2020}{%
	\textup{2020} Mathematics Subject Classification}
\makeatother

\title[SHC of isotropic processes]{Small time asymptotics of spectral heat content of isotropic processes}
\author[Rohan Sarkar]{Rohan Sarkar}
\address{Department of Mathematics\\
	University of Virginia\\
	Charlottesville, VA,  U.S.A.}
\email{kfg3et@virginia.edu}

\keywords{spectral heat content; isotropic processes; self-similarity; asymptotic behavior}
\subjclass[2020]{Primary: 60J45; 60G07; 60G15; 60G51. Secondary: 45K05}
\begin{document}
	\begin{abstract}
		The spectral heat content of a domain $\Omega\subset\R^d$ corresponding to a $d$-dimensional stochastic process $X=(X_t)_{t\ge 0}$ is defined as \[Q^{X}_\Omega(t)=\int_{\Omega} \P_x(\tau^X_\Omega>t)dx,\] where $\tau^X_\Omega$ is the first exit time of $X$ from $\Omega$. The process $X$ can be non-Markovian in general. We provide a novel technique for proving small time asymptotic of spectral heat content for any translation invariant isotropic process satisfying negligible tail probability condition. As a consequence, we recover several existing results in the context of L\'evy processes and Gaussian processes, and provide spectral heat content asymptotics for a class of non-Markovian processes including fractional stable motions, and $\alpha$-stable L\'evy processes time-changed by right inverse of positive, increasing, self-similar Markov processes. The latter has connection to some Cauchy problems that are non-local in both time and space.
	\end{abstract}
	
	\maketitle
	\section{Introduction and main result}
	\subsection{Main result}
	Let $X=(X_t)_{t\ge 0}$ be a stochastic process on $\R^d$ with c\`adl\`ag paths, and let $\P_x$ denotes the probability measure when $X_0=x$.
	For any bounded open set $\Omega\subset \R^d$, the \emph{spectral heat content} (SHC) of $X$ in the domain $\Omega$ is defined as
	\begin{align*}
		Q^X_\Omega(t) =\int_{\Omega} \P_x(\tau^X_\Omega>t) dx,
	\end{align*}
	where $\tau^X_\Omega = \inf\{t\ge 0: X_t\notin \Omega\}$. Since $X$ has c\`adl\`ag paths, it follows that 
	\begin{align*}
		\lim_{t\downarrow 0} Q^X_{\Omega}(t) =\Vol(\Omega),
	\end{align*}
	where $\Vol(\Omega)$ denotes the volume of $\Omega$. In this article, we consider translation invariant isotropic processes $X=(X_t)_{t\ge 0}$ on $\R^d$, that is, for any $x\in\R^d$,
	\begin{align*}
		\mbox{$\P_x$--law of $(X_t)_{t\ge0}$}&=\mbox{$\P_0$--law of $(X_t+x)_{t\ge 0}$}, \\
		\mbox{$\P_0$--law of $(X_t)_{t\ge 0}$}&=\mbox{$\P_0$--law of $(UX_t)_{t\ge 0}$}. 
	\end{align*}
	for all $d\times d$ orthogonal matrices $U$. We do not assume any Markovian structure of the process. Writing $X_t=(X^{(1)}_t, \ldots, X^{(d)}_t)$,
	for any $t>0$, we define
	\begin{equation}\label{eq:mu_sigma}
		\begin{aligned}
			\mu(t)&=\E_0\left[\sup_{0\le s\le t}X^{(1)}_s\wedge 1\right].
		\end{aligned}
	\end{equation}
	From right continuity of the sample paths we have
	\begin{align*}
		\lim_{t\downarrow 0} \mu(t)=0.
	\end{align*}
	Throughout the article, we assume the negligibility of the tail probability in the following sense.
	\begin{assumption}\label{assump1}
		For any $\varepsilon>0$,
		\begin{align*}
			\lim_{t\downarrow 0} \frac{1}{\mu(t)}\P_0\left(\sup_{0\le s\le t} X^{(1)}_s>\varepsilon\right)=0.
		\end{align*}
	\end{assumption}
	The aim of this paper is to prove the following asymptotic behavior of the spectral heat content of $X$.
	\begin{theorem}\label{thm:main} Let $\Omega$ be a bounded open subset of $\R^d$ with $C^{1,1}$ boundary, and $X=(X_t)_{t\ge 0}$ be a translation invariant isotropic stochastic process satisfying Assumption~\ref{assump1}. Then,
		\begin{align}\label{eq:spectral_heat_asymp}
			\lim_{t\downarrow 0}\frac{\Vol(\Omega)-Q^X_\Omega(t)}{\mu(t)}=\Per(\Omega),
		\end{align}
		where $\Per(\Omega)$ denotes the perimeter of $\Omega$.
	\end{theorem} 
	\begin{remark}
		The $C^{1,1}$ regularity condition on the boundary of $\Omega$ inherits from the differentiability of the signed distance function from the boundary, which is crucial in our proof of the above theorem. We refer to Remark~\ref{rem:C11} for details.
	\end{remark}
	Assumption~\ref{assump1} is satisfied by a wide range of isotropic processes. The next result provides sufficient conditions for Assumption~\ref{assump1} to hold in terms of the moments of the supremum process $\sup_{0\le s\le t} X^{(1)}_s$. For the ease of writing, we introduce the following notation to be used throughout the rest.
	\begin{notation}
		For any stochastic process $Z=(Z_t)_{t\ge 0}$ on $\R$, we denote the running supremum of $Z$ by $\overline{Z}$, that is,
		\begin{align*}
			\overline{Z}_t=\sup\{Z_s: 0\le s\le t\}.
		\end{align*}
	\end{notation}
	\begin{proposition}\label{corr1}
		Let $\Omega$ be a bounded, open set with $C^{1,1}$ boundary and $X=(X_t)_{t\ge 0}$ be a translation invariant isotropic process such that 
		\begin{align*}
			\E_0\left[\left(\overline{X}^{(1)}_T\right)^p\right]<\infty
		\end{align*}
		for some $T>0$ and $p>1$, and 
		\begin{align*}
			\lim_{t\downarrow 0}\frac{\E_0\left[\left(\overline{X}^{(1)}_t\right)^p\right]}{\E_0\left[\overline{X}^{(1)}_t\right]}=0.
		\end{align*}
		Then, 
		\begin{align*}
				\lim_{t\downarrow 0}\frac{\Vol(\Omega)-Q^X_\Omega(t)}{\E_0\left[\overline{X}^{(1)}_t\right]}=\Per(\Omega).
		\end{align*}
	\end{proposition}
	\subsection{Previous works} The spectral heat content in the classical sense is defined via the heat equation with Dirichlet boundary condition on $\Omega$. More precisely, if $u(t,x)$ solves the PDE (in the weak sense)
	\begin{align*}
		\frac{\partial}{\partial t} u(x,t)&=\Delta u(x,t) \quad \mbox{in $\Omega\times (0,\infty)$} \\
		u(x,t)&= 0 \quad \mbox{for $(x,t)\in \partial\Omega\times (0,\infty) $} \\
		u(x,0) &= 1 \quad \mbox{for $x\in\Omega$},
	\end{align*}
	where $\Delta$ denotes the Laplace operator on $\R^d$,
	the spectral heat content of $\Omega$ is given by
	\begin{align*}
		Q_\Omega(t)=\int_\Omega u(x,t) dx.
	\end{align*}
	From standard theory of Markov processes, one can write $Q_\Omega(t)=Q^B_\Omega(t)$, where $B$ is the standard Brownian motion on $\R^d$.
	Small time asymptotics of $Q^B_\Omega(t)$ was first proved by van den Berg and Davies, see \cite[Theorem~6.2]{BergDavies1989}, for domains with $R$-smooth boundary, and subsequently, van den Berg and LeGall \cite[Theorem~1.2]{BergLeGall1994} obtained a second order expansion of $Q^B_\Omega(t)$ for domains with $C^3$ boundary as follows:
	\begin{align*}
		Q^B_\Omega(t)=\Vol(\Omega)-\sqrt{\frac{2t}{\pi}}\sigma(\partial\Omega)+\frac{t}{4}\int_{\partial\Omega} H_\Omega d\sigma +\mathrm{O}(t^{3/2}),
	\end{align*}
	where $\sigma$ is the surface measure on $\partial\Omega$, and $H_\Omega$ is the mean curvature at the boundary. The above formula was further generalized to Riemannian manifolds by van den Berg and Gilkey \cite{BergGilkey1994} for domains with smooth boundary. Using PDE techniques, Savo \cite{Savo1999} proved an asymptotic expansion of $Q_\Omega(t)$ of arbitrary order for any domain with smooth boundary in a Riemannian manifold.
	
	While the asymptotics for Brownian motion is known for several years, results for jump processes have been obtained quite recently. In the one-dimensional case, SHC asymptotic for symmetric $\alpha$-stable L\'evy process was obtained by Acu\~na Valverde \cite{Valverde2016}, and later generalized to isotropic $\alpha$-stable L\'evy processes on $\R^d$ by Park and Song \cite{ParkSong2022}. We refer to \cite{KobayashiKeiPark2024} for SHC of isotropic L\'evy processes with regularly varying L\'evy-Khintchine exponents, and to a recent article \cite{LeePark2025} for symmetric L\'evy processes (may not be isotropic) with some conditions on the L\'evy-Khintchine exponent. When the L\'evy process has bounded variation, the small time asymptotic of SHC was proved in \cite{GrzywnyParkSong2019}, and in this case, one recovers the fractional perimeter in the limit. Beyond L\'evy processes, we refer to \cite{KobayashiKeiPark2025, KobayashiPark2023, KobayashiPark2022} for separate treatments of SHC asymptotics of time-changed killed Brownian motion, time-changed killed $\alpha$-stable processes, and Gaussian processes in $\R^d$.

	\subsection{Our contribution} In the present article, we propose a novel approach based on simple facts from geometric measure theory for proving SHC asymptotic for any translation invariant isotropic process satisfying Assumption~\ref{assump1}. This assumption holds true in particular for any isotropic L\'evy process with regularly varying L\'evy-Khintchine exponent, and Gaussian processes. Moreover, in Section~\ref{sec:time_changed_BM}, we provide a wide range of examples of time-changed isotropic $\alpha$-stable processes that satisfy Assumption~\ref{assump1}. As a consequence of Theorem~\ref{thm:main}, we recover the results proved in \cite{ParkSong2022, KobayashiPark2022, KobayashiPark2023, KobayashiKeiPark2024}. 
	
	Our result is also applicable to \emph{fractional stable motions}, which are generalizations of fractional Brownian motions, and such processes were introduced by Samorodnitsky and Taqqu \cite{SamorodnitskyTaqquBook}. We refer to Section~\ref{sec:fractional} for details. 
	
	We also provide examples of time-changed $\alpha$-stable L\'evy processes for which the SHC can be related to the (weak) solution of some integro-differential equations which are non-local in both time and space. Examples of such stochastic processes can be found in Section~\ref{sss:self_similar}. 
	
	An important aspect of our method is that it requires very minimal assumption on the isotropic process in terms of the tail probability behavior of the supremum of the one dimensional marginals. This observation led us to adapt this technique for heat content of fractional sub-Laplacian on Carnot groups and we refer the interested reader to \cite{Sarkar2026}.
	
	Our techniques can be applied to bounded domains with boundary being level sets of a $C^{1,\kappa}$ functions for $0<\kappa<1$. In all the aforementioned references \cite{BergDavies1989,BergGilkey1994,BergLeGall1994,ParkSong2022,KobayashiKeiPark2024,KobayashiKeiPark2025}, the assumption of $C^{1,1}$ regularity is crucial as the proofs rely heavily on the uniform interior and exterior ball condition of the domain. In fact, from \cite[Lemma~2.2]{Aikawa_et_al2007} (see also \cite[Corollary~2]{LewickaPeres2020}), the uniform interior and exterior ball condition is equivalent to the $C^{1,1}$ smoothness of the boundary. In presence of $C^{1,1}$ regularity of the boundary, one can parametrize the tubular neighborhoods of the boundary by $C^{1,1}$ functions, which is used to obtain the SHC asymptotics. For $C^{1,\kappa}$ boundary with $0<\kappa<1$, such parametrization does not exist. In Theorem~\ref{thm:main2} we provide upper and lower bound of the SHC for domains with boundaries given by level sets of $C^{1,\kappa}$ functions with $0<\kappa<1$. To the best of our knowledge, such estimates have not been derived before for general isotropic processes. On the other hand, SHC for Brownian motion has been studied for domains with non-smooth boundaries: see \cite[Corollary~1.4]{Savo1999} for general convex domains; see \cite{vandenBerg1988} for $\partial\Omega$ with polygonal boundary; see \cite{vandenBerg1998} for $\partial\Omega$ with fractal polygonal boundary. In general, the question of validity of \eqref{eq:spectral_heat_asymp} (even for Brownian motions) for domains with $C^{1,\kappa}$ boundary still remains open.
	
	Let us now briefly describe the main idea behind our approach to prove Theorem~\ref{thm:main}. For any nonnegative integrable function $f:\R^d\to\R$ we define
	\begin{align}\label{eq:functional_spectral_heat}
		Q^X_f(t)=\int_{\R^d} \E_x\left[\inf_{0\le s\le t} f(X_s)\right] dx.
	\end{align}
	In particular, when $f=\mathbbm{1}_\Omega$, $Q^X_f(t)=Q^X_\Omega(t)$ for all $t\ge 0$. Using Taylor expansion of $f$, we first obtain the asymptotic of $Q^X_f$ when $f$ is a compactly supported smooth function, and the lower bound of the limit in \eqref{eq:spectral_heat_asymp} is proved using an approximation argument. It turns out that the lower bound of \eqref{eq:spectral_heat_asymp} holds for any bounded, open subset with finite perimeter, see Proposition~\ref{prop:liminf}. The key idea to derive the upper bound of the limit in \eqref{eq:spectral_heat_asymp} stems from the observation that any bounded domain with $C^{1,1}$ boundary can be expressed as $\phi^{-1}(0,\infty)$ for some bounded $C^{1,1}$ function $\phi$ with $|\nabla \phi|$ being constant on the boundary.  
	
	The rest of the paper is organized as follows: in Section~\ref{sec:2} we provide examples including isotropic L\'evy processes, Gaussian processes, fractional stable motions, and a class of time-changed $\alpha$-stable isotropic L\'evy processes for which spectral heat content asymptotics follow directly from our main result. Section~\ref{sec:3} provides some preliminaries about functions of bounded variation and regularity of boundary. In Section~\ref{sec:4}, we prove the asymptotic of the functional of the spectral heat content defined in \eqref{eq:functional_spectral_heat}. Section~\ref{sec:lower_bound} establishes the lower bound for the spectral heat content asymptotics using an argument involving convolution with mollifiers. Finally, Theorem~\ref{thm:main} is proved in Section~\ref{sec:6}.

	\section{Examples}\label{sec:2}
	\subsection{Isotropic L\'evy processes} We consider the class of L\'evy processes studied in \cite{KobayashiKeiPark2024}. Let $X=(X_t)_{t\ge 0}$ be an isotropic L\'evy process with L\'evy-Khintchine exponent $\Psi(\xi)=\psi(|\xi|)$ for any $\xi\in\R^d$, and assume that $\psi:[0,\infty)\to [0,\infty)$ is regularly varying with exponent $\beta\in (1,2]$. Then, by \cite[Equation~(3.2)]{Pruitt1981}, for any $\epsilon>0$, 
	\begin{align*}
		\P_0\left(\overline{X}^{(1)}_t>\epsilon\right)\le c_\epsilon t
	\end{align*}
	for some constant $c_\epsilon>0$. On the other hand, using the bounds in \cite[Equation~(3.8)]{LeePark2025}, we have $\lim_{t\downarrow 0}\frac{t}{\mu(t)}=0$,
	which shows that Assumption~\ref{assump1} holds. Therefore, Theorem~\ref{thm:main} holds for $X$, and we recover \cite[Theorem~2.1(1)]{LeePark2025} for isotropic L\'evy processes. In particular, when $X$ is the isotropic $\alpha$-stable L\'evy process, that is, $\Psi(\xi)=|\xi|^\alpha$, as $t\downarrow 0$ we have
	\begin{align*}
		\mu(t)\sim\begin{cases}
			t^{1/\alpha}\E[\overline{X}^{(1)}_{1}] & \mbox{if $1<\alpha\le 2$},\\
			\frac{1}{\pi}t\log(1/t) & \mbox{if $\alpha=1$}.
		\end{cases}
	\end{align*}
	For $\alpha\in (1,2]$, the asymptotic behavior of $\mu(t)$ above follows directly from the self-similarity of $\alpha$-stable process, and for $\alpha=1$, we refer to \cite[Proposition~4.3(i)]{Valverde2016}.
	This recovers \cite[Theorem~1.1]{ParkSong2022} when $\alpha\ge 1$.
	\subsection{Gaussian processes} Consider a translation invariant Gaussian process $X=(X_t)_{t\ge 0}$ on $\R^d$ with independent and identically distributed (i.i.d.) components under the probability law $\P_0$, and assume that $X$ has c\`adl\`ag sample paths. Also assume that $\E_0[X_t]=0$ for all $t\ge 0$. In such a case, $X$ is an isotropic process. Indeed, for any $i,j=1,\ldots, d$ and $s,t>0$, 
	\begin{align*}
		\mathrm{Cov}(X^{(i)}_s, X^{(j)}_t)=\delta_{ij} R(s,t), \quad R(s,t):=\mathrm{Cov}(X^{(1)}_s, X^{(1)}_t),
	\end{align*}
	where \[\delta_{ij}=\begin{cases}
		1 & \mbox{if $i=j$} \\
		0 & \mbox{if $i\neq j$}.
	\end{cases}
	\]
	Then for any orthogonal matrix $U$ we have
	\begin{align*}
		\mathrm{Cov}(UX_s, UX_t)=U\mathrm{Cov}(X_s,X_t) U^\top=R(s,t)UU^\top=R(s,t)I_{d\times d}.
	\end{align*}
	This shows that $X$ and $UX$ have the same covariance kernel, and therefore they have same distribution as Gaussian processes.
	
	For any $t>0$, let 
	\begin{align*}
		m(t)=\E_0\left[\overline{X}^{(1)}_t\right], \quad 
		\sigma(t)^2=\sup_{0\le s\le t}\mathrm{Var}\left[X^{(1)}_s\right].
	\end{align*}
	
	We assume that $m(t)<\infty$ for all $t\in [0,1]$, which is known to be equivalent to the condition that $\P_0(\overline{X}^{(1)}_1<\infty)=1$. 
By Borell-TIS inequality, see \cite[Theorem~2.1.1]{AdlerTaylorBook}, for any $x>0$ 
	\begin{equation}\label{eq:BTIS}
		\begin{aligned}
			\P_0\left(\left|\overline{X}^{(1)}_t-m(t)\right|>x\right)&\le 2 e^{-x^2/2\sigma(t)^2},
		\end{aligned}
	\end{equation}
	and hence 
	\begin{align*}
		\mathrm{Var}\left[\overline{X}^{(1)}_t\right]\le\int_0^\infty 4x e^{-x^2/2\sigma(t)^2}=4\sigma(t)^2.
	\end{align*}
	Therefore, $\E_0\left[\left(\overline{X}^{(1)}_t\right)^2\right]\le m(t)^2+4\sigma(t)^2$.
	From \cite[Equation~(29)]{KobayashiKeiPark2025} it is known that $\lim_{t\downarrow 0}{\sigma(t)^2}/{m(t)}=0$. 
	Applying Proposition~\ref{corr1}, we obtain the following result.
	\begin{proposition}\label{prop:gaussian_assymp}
		Let $(X_t)_{t\ge 0}$ be an isotropic Gaussian process with c\`adl\`ag paths. If $\Omega$ is a bounded open set with $C^{1,1}$ boundary then,
		\begin{align}\label{eq:gaussian_asymp}
			\lim_{t\downarrow 0}\frac{\Vol(\Omega)-Q^X_\Omega(t)}{m(t)}=\Per(\Omega).
		\end{align}
	\end{proposition}
	\begin{remark}
		In \cite[Theorem~4.4]{KobayashiKeiPark2025} under the condition that the rescaled process $Y^{(t)}_u=\mu^{-1}_t X^{(1)}_{tu}$ converges weakly to a Gaussian process $(Y_u)_{u\ge 0}$ as $t\downarrow 0$ in Skorokhod $J_1$ topology, it is proved that
		\begin{align*}
			\lim_{t\downarrow 0}\frac{\Vol(\Omega)-Q^X_\Omega(t)}{m(t)}=\E\left[\overline{Y}_{\!1}\right]\Per(\Omega),
		\end{align*}
		with $\E[\overline{Y}_{\!1}]\le 1$. Proposition~\ref{prop:gaussian_assymp} does not require the convergence of the rescaled Gaussian process, and hence it strengthens \cite[Theorem~4.4]{KobayashiKeiPark2025}.
	\end{remark}
	
	\subsection{Fractional stable motions}\label{sec:fractional} Let $M_\alpha$ be an $\alpha$-stable $\R^d$--valued isotropic random measure defined as follows:
	\begin{enumerate}
		\item For pairwise disjoint subsets $A_1,A_2,\ldots\in\mathcal{B}(\R)$,
		\begin{align*}
			M_\alpha(\cup_{i=1}^\infty A_i)=\sum_{i=1}^\infty M_\alpha(A_i) \quad \mbox{almost surely},
		\end{align*}
		\item for each $A\in \mathcal{B}(\R)$ with $|A|<\infty$, $M_\alpha(A)$ is an isotropic $\alpha$-stable random variable, that is, 
		\begin{align*}
			\E\left[e^{{\rm{i}}\langle\lambda, M_\alpha(A)\rangle} \right]=\exp\left(-|\lambda|^\alpha|A|\right) \quad \text{for all $\lambda\in\R^d$},
		\end{align*}
		where $|A|$ is the Lebesgue measure of $A$, and
		\item for disjoint subsets $A_1,\ldots, A_n$ in $\mathcal{B}(\R)$, $M_\alpha(A_1), \ldots, M_\alpha(A_n)$ are independent.
	\end{enumerate}
	Such random measure can be constructed similarly to the case $d=1$, which can be found in the book by Samorodnitsky and Taqqu \cite[Section~3.3]{SamorodnitskyTaqquBook}. Let $f:\R_+\times\R\to \R$ be a function such that for some $T>0$, 
	\begin{align}\label{eq:int_condition}
		\int_\R |f(t,x)|^\alpha dx<\infty \quad \text{for all $t\in [0,T]$}.
	\end{align}
The \emph{fractional stable motion} with kernel $f$ is the stochastic process $X=(X_t)_{t\in [0,T]}$ defined by
	\begin{align}\label{eq:stable_def}
		X_t=\int_\R f(t,x) M_\alpha(dx).
	\end{align}
	Due to the integrability condition \eqref{eq:int_condition}, $X_t$ is well defined for each $t\in [0,T]$, see \cite[Section~3.4]{SamorodnitskyTaqquBook} for details. By construction, $X$ is $\alpha$-stable. More precisely, for any $t_1,\ldots, t_n\in [0,T]$ and $\lambda_1,\ldots,\lambda_n\in\R^d$,
	\begin{align*}
		\E\left[e^{{\rm{i}}\sum_{j=1}^n\langle\lambda_j,X_{t_j}\rangle}\right]=\exp\left(-\int_{\R}\left|\sum_{j=1}^n\lambda_jf(t_j,x)\right|^\alpha dx\right).
	\end{align*}
	Moreover, $X$ is an isotropic process. Indeed, for any $d\times d$ orthogonal matrix $U$, $t_1,\ldots, t_n\in [0,T]$, and $\lambda_1,\ldots, \lambda_n\in\R^d$,
	\begin{align*}
			&\ \ \ \ \E\left[e^{{\rm{i}}\sum_{j=1}^n\langle\lambda_j,UX_{t_j}\rangle}\right] \\
			&=\E\left[e^{{\rm{i}}\sum_{j=1}^n\langle U^\top \lambda_j,X_{t_j}\rangle}\right] \\
			&=\exp\left(-\int_{\R}\left|U^\top \sum_{j=1}^n\lambda_jf(t_j,x)\right|^\alpha dx\right) \\
			&=\exp\left(-\int_{\R} \left|\sum_{j=1}^n\lambda_jf(t_j,x)\right|^\alpha dx\right) \\
			&=	\E\left[e^{{\rm{i}}\sum_{j=1}^n\langle\lambda_j,X_{t_j}\rangle}\right]
	\end{align*}

	For $H\in (0,1)$ and $\alpha\in (0,2]$, $X$ is called 
	\begin{itemize}[leftmargin=*]
	\item \emph{linear fractional stable motion (LFSM)} if 
	\begin{align*}
		f(t,x)=f_1(t,x)=(t-x)^{H-1/\alpha}_+ - (-x)^{H-1/\alpha}_+,
	\end{align*}
	\item \emph{balanced linear fractional stable motion (BLFSM)} if 
	\begin{align*}
		f(t,x)=f_2(t,x)=|t-x|^{H-1/\alpha} - |x|^{H-1/\alpha},
	\end{align*}
	\item \emph{harmonizable fractional stable motion (HFSM)} if 
	\begin{align*}\label{eq:f_4}
		X_t=\int_\R f_3(t,x) M_{1,\alpha}(dx)-\int_\R f_4(t,x) M_{2,\alpha}(dx)
	\end{align*}
	where $(M_{1,\alpha}, M_{2,\alpha})$ is an $\R^{2d}$--valued $\alpha$-stable isotropic random measure, and 
	\begin{align*}
		 f_3(t,x)=\frac{\cos(tx)-1}{|x|^{H+1/\alpha}}, \quad f_4(t,x)=\frac{\sin(tx)}{|x|^{H+1/\alpha}}.
	\end{align*}
	\end{itemize}
	It can be verified that $f_i$ satisfies \eqref{eq:int_condition} for $i=1,2,3,4$ and for all $t\ge 0$, and $X$ is self-similar with index $H$. Moreover, all these processes have stationary increments. In particular when $\alpha=2$, both LFSM and HFSM reduce, up to multiplicative normalizing constants, to the fractional Brownian motion with Hurst parameter $H$.
	
	When $\alpha\in (1,2]$ and $H\in (1/\alpha,1)$, both LFSM and BLFSM have H\"older continuous sample paths, as proved in \cite{Takashima1989}. In fact, when $\alpha\in (1,2)$ and $H<1/\alpha$, it can be shown that the sample paths of LFSM/BLFSM are unbounded in any compact interval almost surely, and therefore cannot have continuous paths, see \cite[Example~10.2.5]{SamorodnitskyTaqquBook}. K\^ono and Maejima \cite[Theorem~1]{KonoMaejima1991} showed that HFSM has H\"older continuous sample paths for any $\alpha\in (0,2)$ and $H\in (0,1)$. This also indicates that for $\alpha<2$, LFSM and HFSM are different processes.
	\begin{proposition}\label{prop:fractional}
	Let $\alpha\in (1,2]$ and $X=(X_t)_{t\ge 0}$ be either an LFSM/BLFSM with $H\in (1/\alpha,1)$, or an HFSM with $H\in (0,1)$. Assume that $(X,\P_x)_{x\in\R^d}$ is translation invariant. Then, for any bounded open set $\Omega$ with $C^{1,1}$ boundary,
		\begin{align*}
			\lim_{t\to 0}\frac{\Vol(\Omega)-Q^X_\Omega(t)}{t^H}=\E_0\left[\overline{X}^{(1)}_1\right]\Per(\Omega).
		\end{align*}
	\end{proposition}

	\begin{proof}
		Since $X$ has continuous sample paths, for any $T>0$, $\sup_{0\le t\le T} |X_t|<\infty$ almost surely. Using \cite[Theorem~10.5.1]{SamorodnitskyTaqquBook}, there exists a finite number $K\ge 0$ such that
		\begin{align*}
			\lim_{x\to\infty} x^\alpha\P\left(|\overline{X}^{(1)}_T|>x\right)=K.
		\end{align*}
		This shows that $\E[|\overline{X}^{(1)}_T|^p]<\infty$ whenever $0\le p<\alpha$, and by $H$--self-similarity of $X$, 
		\begin{align*}
			\E_0\left[\left(\overline{X}^{(1)}_t\right)^p\right]=t^{pH} \E_0\left[\left(\overline{X}^{(1)}_1\right)^p\right].
		\end{align*}
		Since $\alpha>1$, we may choose $p\in (1,\alpha)$, and therefore, the proof of the proposition follows by directly invoking Proposition~\ref{corr1}.
	\end{proof}
	\subsection{Time-changed $\alpha$-stable L\'evy process}\label{sec:time_changed_BM} Consider the $d$-dimensional $\alpha$-stable isotropic L\'evy process $Y=(Y_t)_{t\ge 0}$ with $\alpha\in (1,2]$, and a nonnegative, increasing process $U=(U_t)_{t\ge 0}$ with c\`adl\`ag paths which is independent of $Y$ and $U_0=0$. Define $X_t=Y_{U_t}$. Then, $X=(X_t)_{t\ge 0}$ is translation invariant and isotropic. Let us assume that for some $p>1$,
	\begin{equation}\label{eq:U_conditions}
		\begin{aligned}
			&\E[U^{p/\alpha}_t]<\infty \quad \text{for all $t\ge 0$},  \quad \lim_{t\downarrow 0}\frac{\E[U^{p/\alpha}_t]}{\E[U^{1/\alpha}_t]}=0.
		\end{aligned}
	\end{equation}
	Then we have the following result.
	\begin{proposition}\label{prop:U_limits}
		Let $X_t=Y_{U_t}$ where $(Y_t)_{t\ge 0}$ is an isotropic $\alpha$-stable L\'evy process with $\alpha\in (1,2]$, and $(U_t)_{t\ge 0}$ is a nonnegative, increasing process independent of $(Y_t)_{t\ge 0}$ satisfying \eqref{eq:U_conditions}. Then for any bounded open set $\Omega$ with $C^{1,1}$ boundary we have
		\begin{align}\label{eq:U_limit}
			\lim_{t\downarrow 0}\frac{\Vol(\Omega)-Q^X_\Omega(t)}{m(t)}=\Per(\Omega),
		\end{align}
		where $m(t)=\E_0\left[\overline{X}^{(1)}_t\right]$. In particular when $U_t$ has continuous paths, we have $m(t)=\E[\overline{Y}^{(1)}_1]\E[U^{1/\alpha}_t]$.
	\end{proposition}
	\begin{remark}
		In the context of time-changed L\'evy processes, SHC of the time-changed process and the killed time-changed process are different in general. This happens as the paths of $U$ are discontinuous in general. However when $U$ has continuous paths, above two notions of SHC coincide since
	\begin{align*}
		\P_x\left(\tau^X_\Omega>t\right)=\P_x\left(\tau^Y_\Omega>U_t\right)
	\end{align*}
	for all $x\in\Omega$ and $t\ge 0$. We refer to \cite[p.~1152]{KobayashiPark2023} (see also \cite[p.~580]{SongVondracek2003}) for detailed discussions on the comparison between killed time-changed and time-changed killed processes.
	\end{remark}
	\begin{proof} 
		Since $Y$ is an $\alpha$-stable L\'evy process, for any $0\le q<\alpha$, $\E[\sup_{0\le s\le t}|Y_t|^q]<\infty$, see for instance \cite[Theorem~25.18]{Sato_Book}. By $1/\alpha$--self-similarity of $Y$, we have $\E[(\overline{Y}^{(1)}_t)^q]= t^{q/\alpha}\E[(\overline{Y}^{(1)}_1)^q]$. Since $U$ is increasing and right continuous, for any $t\ge 0$,
		\begin{align*}
			\overline{X}^{(1)}_t \le \sup_{0\le s\le U_t} Y^{(1)}_{s},
		\end{align*}
		which shows that 
		\begin{align}\label{eq:e_bound1}
			\E\left[\left(\overline{X}^{(1)}_t\right)^q\right]\le \E_0\left[\left(\overline{Y}^{(1)}_1\right)^q\right]\E\left[U^{q/\alpha}_t\right]<\infty
		\end{align}
		for $0\le q<p\wedge\alpha$.
		On the other hand, using symmetry of the distribution of $X^{(1)}$, we note that for any $u>0$,
		\begin{align*}
			\P_0\left(\overline{X}^{(1)}_t\ge u\right)\ge \frac{1}{2}\P_0\left(|X^{(1)}_t|\ge u\right),
		\end{align*}
		which implies that for any $0\le q<p\wedge\alpha$,
		\begin{align}\label{eq:e_bound2}
			\E_0\left[\left(\overline{X}^{(1)}_t\right)^q\right]\ge \frac{1}{2}\E_0\left[|X^{(1)}_t|^q\right]=\frac{1}{2}\E\left[|Y^{(1)}_1|^q\right]\E[U^{q/\alpha}_t],
		\end{align}
		where the last equality follows from self-similarity of $Y$. Therefore, combining \eqref{eq:e_bound1} and \eqref{eq:e_bound2}, for any $t\ge 0$ and $0\le q<p\wedge\alpha$ we have
		\begin{align*}
		 \frac{1}{2}\frac{\E_0\left[|Y^{(1)}_1|^q\right]}{\E_0\left[\overline{Y}^{(1)}_1\right]}\frac{\E[U^{q/\alpha}_t]}{\E[U^{1/\alpha}_t]} \le 	\frac{\E_0\left[\left(\overline{X}^{(1)}_t\right)^q\right]}{\E_0\left[\overline{X}^{(1)}_t\right]}\le 2\frac{\E_0\left[\left(\overline{Y}^{(1)}_1\right)^q\right]}{\E_0\left[|Y^{(1)}_1|\right]} \frac{\E[U^{q/\alpha}_t]}{\E[U^{1/\alpha}_t]}.
		\end{align*}
		
		In light of Proposition~\ref{corr1}, it suffices to show that $\E[U^{q/\alpha}_t]/\E[U^{1/\alpha}_t]\to 0$ whenever $1<q<p\wedge \alpha$, where $p$ is given by \eqref{eq:U_conditions}. We claim that for any nonnegative  random variable $Z$ with $0<\E[Z^p]<\infty$ for some $p>1$, and $1<q<p$,
		\begin{align*}
			\frac{\E[Z^{q}]}{\E[Z]}\le \left(\frac{\E[Z^{p}]}{\E[Z]}\right)^{\frac{q-1}{p-1}}.
		\end{align*}
		Indeed, let us write $a=\frac{p-1}{q-1}$, $b=\frac{p-1}{p-q}$. Then, $a^{-1}+b^{-1}=1$. Also note that
		\begin{align*}
			Z^q=Z^{\frac{p}{a}}\cdot Z^{\frac{1}{b}}.
		\end{align*}
		Therefore by H\"older's inequality we get
		\begin{align*}
			\E[Z^q]\le \E[Z^p]^{\frac{1}{a}}\E[Z]^{\frac{1}{b}},
		\end{align*}
		which directly proves our claim. Choosing $Z=U_t^{1/\alpha}$, we therefore obtain
		\begin{align*}
			\frac{\E[U^{q/\alpha}_t]}{\E[U^{1/\alpha}_t]}\le \left(\frac{\E[U^{p/\alpha}_t]}{\E[U^{1/\alpha}_t]}\right)^{\frac{q-1}{p-1}},
		\end{align*}
		which shows that the conditions of Proposition~\ref{corr1} hold due to \eqref{eq:U_conditions}. This concludes the proof of \eqref{eq:U_limit}.

		 When $U$ has continuous paths, one has
		\begin{align*}
			\sup_{0\le s\le t} Y^{(1)}_{U_s}=\sup_{0\le s\le U_t} Y^{(1)}_{s} \quad \mbox{$\P_0$--almost surely}.
		\end{align*}
		By independence of $Y$ and $U$, and the self-similarity of $Y$, we conclude that $m(t)=\E[\overline{Y}^{(1)}_1]\E[U^{1/\alpha}_t]$.
		\end{proof}
		\subsubsection{Time-change by inverse subordinator} Consider a subordinator (increasing L\'evy process on $(0,\infty)$, see Chapter III in Bertoin \cite{BertoinBook}) $S=(S_t)_{t\ge 0}$ with $S_0=0$ such that 
		\begin{align*}
			\E[e^{-\lambda S_t}]=e^{-t\varphi(\lambda)} \quad \mbox{for all $\lambda, t>0$},
		\end{align*}
		where $\varphi$ is a Bernstein function such that 
		\begin{align}\label{eq:Bernstein}
			\varphi(\lambda)=b\lambda+\int_0^\infty(1-e^{-\lambda r})\mu(dr)
		\end{align}
		with $b\ge 0$, and $\int_{0}^\infty\min\{1,r\}\mu(dr)<\infty$.
		 The inverse subordinator $E=(E_t)_{t\ge 0}$ is defined as the right inverse of $(S_t)_{t\ge 0}$, that is, 
		 \begin{align*}
		 	E_t=\inf\{s\ge 0: S_s>t\}.
		 \end{align*}
		 Then, $E$ is also an increasing process, and it is continuous in time if and only if $S$ has strictly increasing sample paths almost surely, which occurs if and only if $\varphi(\infty)=\infty$.
		 \begin{proposition}
		 	Assume that $\varphi$ is regularly varying at infinity with index $\beta\in (0,1)$ and $X_t=Y_{E_t}$, where $Y$ is a $d$-dimensional isotropic $\alpha$-stable L\'evy process with $\alpha\in (1,2]$ and independent of $S$. Then,
		 	\begin{align*}
		 		\lim_{t\downarrow 0}\varphi(1/t)^{\frac 1\alpha}(\Vol(\Omega)-Q^X_\Omega(t))=\frac{\Gamma(1+1/\alpha)}{\Gamma(1+\beta/\alpha)}\E_0\left[\overline{Y}^{(1)}_1\right]\Per(\Omega).
		 	\end{align*}
		\end{proposition}
		\begin{remark} The above proposition has been proved in \cite[Theorem~4.4]{KobayashiPark2023} for $\alpha=2$, and in \cite{KobayashiPark2022} for $\alpha\in (1,2)$. Below we provide an alternative argument using Proposition~\ref{prop:U_limits}.
		\end{remark}
		\begin{remark}
			When $\varphi(\lambda)=\lambda^\beta$, $\beta\in (0,1)$ and $\alpha=2$, $Q^{X}_\Omega(t)=\int_\Omega u_\beta(x,t)dx$, where $u_\beta$ is the unique solution to the time-fractional Cauchy problem with Dirichlet boundary condition on $\Omega$, see \cite{MeerschaertNaneVellaisamy2009} for details.
		\end{remark}
		\begin{proof} When $\varphi$ is regularly varying at $\infty$, $\varphi(\infty)=\infty$. Therefore from the discussion above, $E$ has continuous sample paths.
			 By \cite[Proposition~4.2]{KobayashiPark2023}, for any $q>0$ one has
			\begin{align}\label{eq:E_t_estimate}
				\E\left[E^q_t\right]\sim \frac{\Gamma(q+1)}{\Gamma(q\beta+1)}\varphi(1/t)^{-q} \quad \mbox{as $t\downarrow 0$}.
			\end{align}
			Hence, for any $p>1$, 
			\begin{align*}
				\frac{\E[E^{p/\alpha}_t]}{\E[E^{1/\alpha}_t]}\sim \frac{\Gamma(p/\alpha+1)\Gamma(\frac{\beta}{\alpha}+1)}{\Gamma(1+1/\alpha)\Gamma(p\beta/\alpha+1)}\frac{\varphi(1/t)^{1/\alpha}}{\varphi(1/t)^{p/\alpha}} \quad \mbox{as $t\downarrow 0$},
			\end{align*}
			which shows that $\lim_{t\downarrow 0}\E[E^{p/\alpha}_t]/\E[E^{1/\alpha}_t]= 0$, since $\varphi$ is regularly varying at $\infty$. Therefore using Proposition~\ref{prop:U_limits} and \eqref{eq:E_t_estimate} for $q=1/\alpha$, we conclude the proof.
		\end{proof}
		\subsubsection{Time-change by any positive, increasing self-similar process}\label{sss:self_similar} Let $U=(U_t)_{t\ge 0}$ be any positive, $\beta$--self-similar, increasing process with c\`adl\`ag paths such that $\E[U^{p/\alpha}_1]<\infty$ for some $p>1$, and $Y$ is same as before. Then, $X_t=Y_{U_t}$ is $\beta/\alpha$--self-similar, that is,
		\begin{align*}
			(cX_t)_{t\ge 0}\overset{d}{=} (X_{tc^{\frac{\alpha}{\beta}}})_{t\ge 0} \quad \mbox{for all $c>0$}.
		\end{align*}
		By self-similarity, we also have 
		\begin{align*}
			\E[U^{q/\alpha}_t]&=t^{q\beta/\alpha}\E[U^{q/\alpha}_1] \quad \mbox{for $q\in [0,p]$}, \\
			 \E\left[\overline{X}^{(1)}_t\right]&=t^{\frac{\beta}{\alpha}}\E\left[\overline{X}^{(1)}_1\right].
		\end{align*}
		Therefore, by Proposition~\ref{prop:U_limits} we deduce the following result.
		\begin{proposition}\label{prop:time_change_SS}
			Let $X_t=Y_{U_t}$ where $(Y_t)_{t\ge 0}$ is an isotropic $\alpha$-stable L\'evy process with $\alpha\in (1,2]$, and $(U_t)_{t\ge 0}$ is a nonnegative, increasing $\beta$--self-similar process with c\`adl\`ag paths. Then, for any bounded open set $\Omega$ with $C^{1,1}$ boundary,
			\begin{align*}
				\lim_{t\downarrow 0}\frac{\Vol(\Omega)-Q^X_\Omega(t)}{t^{\frac{\beta}{\alpha}}}=\E\left[\overline{X}^{(1)}_1\right]\Per(\Omega).
			\end{align*}
		\end{proposition}
		
		We now provide a specific example of a self-similar increasing process constructed from an increasing self-similar Markov process. Let $\xi=(\xi_t)_{t\ge 0}$ be a positive, increasing, self-similar Markov process with self-similarity index $1/\beta$, that is, for all $c>0$,
		\begin{align*}
			\left(c\xi_{tc^{-\beta}}, \P_x\right)\overset{d}{=}\left(\xi_t, \P_{cx}\right)
		\end{align*}
		 By Lamperti's theorem, see \cite{Lamperti1972}, under the probability law $\P_x$ for any $x>0$, one can write $\xi$ as a time-changed subordinator as follows:
		\begin{align*}
			\xi_t=x\exp\left(S_{A_{tx^{-\beta}}}\right),
		\end{align*}
		where $S$ is a subordinator and $A_t=\inf\{s\ge 0: \int_0^s \exp(\beta S_r)dr>t\}$. Moreover, the above representation gives a bijection between the class of all positive, increasing, self-similar Markov processes and subordinators. Let $\varphi$ denote the Bernstein function associated with $S$ defined as in \eqref{eq:Bernstein}.
		For simplicity, we exclude the case when $S$ is arithmetic, that is, $b=0$ and $\mu$ is supported on a lattice $r\mathbb{Z}$ for some $r>0$. 
		 When $\varphi'(0+)<\infty$, or equivalently, $\E[S_1]<\infty$, Bertoin and Caballero \cite{BertoinCaballero2002} proved that the process $(\xi_t,\P_x)$ converges weakly as $x\downarrow 0$, and therefore, one gets an increasing, self-similar process $\xi$ with $\xi_0=0$. Let $\zeta=(\zeta_t)_{t\ge 0}$ denote the right inverse of $\xi$, that is,
		\begin{align*}
			\zeta_t=\inf\{s>0: \xi_s>t\}.
		\end{align*}
		Then, $\zeta$ is $\beta$-self-similar, positive, and increasing process but may not be Markovian. From \cite[Lemma~2.1]{LoeffenPatieSavov2019} it is known that $\E[\zeta^p_1]<\infty$ for all $p>0$. For a $d$-dimensional $\alpha$-stable isotropic L\'evy process $Y=(Y_t)_{t\ge 0}$ independent of $\xi$, let us define the time-changed process $X_t=Y_{\zeta_t}$. Then, Proposition~\ref{prop:time_change_SS} holds for $X$. We note that $\zeta$ has discontinuous paths whenever $\varphi(\infty)<\infty$.
		
	 Since $\varphi'(0+)<\infty$, using Fubini theorem and a simple change of variable leads to
		\begin{align*}
			\varphi(\lambda)=b\lambda+\lambda\int_0^1 r^{\lambda-1}m(r)dr,
		\end{align*}
	where $m$ is a non-decreasing function on $(0,1)$ and $\int_0^1 (-\log r\wedge 1) m(r)dr<\infty$. Patie and Srapionyan \cite[Definition~3.1]{PatieSrapionian2021} introduced the following non-local operator:
	\begin{align*}
		\partial^\varphi_t f= t^{1-\beta} b f'(t)+ t^{-\beta}\int_0^t f'(r)m(r/t) dr.
	\end{align*}
	In particular, when $\xi$ is the standard $\beta$-stable subordinator with $\beta\in (0,1)$, one has 
	\begin{align*}
		b=0 \quad  \text{and} \quad m(r)=\frac{(1-r)^{-\beta}}{\Gamma(1-\beta)}\mathbbm{1}_{\{0<r<1\}},
	\end{align*}
	 see e.g. \cite[Example~5]{BertoinYor2002}, and in this case $\partial^\varphi_t$ coincides with the Caputo fractional derivative of order $\beta$. Let us assume that $\zeta$ has continuous paths, or equivalently, $\varphi(\infty)=\infty$. Consider the Cauchy problem
	\begin{equation}\label{eq:integro_diff}
	\begin{aligned}
		&\partial^\varphi_t u =\Delta^\alpha_\Omega u \\
		&u(x,0)=\mathbbm{1}_\Omega(x),
	\end{aligned}
	\end{equation}
	where $\Delta^\alpha_\Omega$ is the fractional Laplacian with Dirichlet boundary condition on $\Omega$, defined as the Friedrichs extension of $(-(-\Delta)^{\alpha/2},C^\infty_c(\Omega))$ in $L^2(\Omega)$, and $\alpha\in (1,2]$. Since $\Delta^\alpha_\Omega$ is self-adjoint and generates a compact semigroup in $L^2(\Omega)$, by \cite[Theorem~4.1]{PatieSrapionian2021}, $u_\varphi(x,t)=\P_x\left(\tau^Y_\Omega>\zeta_t\right)$ is a solution to the above integro-partial differential equation \eqref{eq:integro_diff}. Due to path continuity of $\zeta$, we get $u_\varphi(x,t)=\P_x\left(\tau^X_\Omega>t\right)$ for all $x\in\Omega$ and $t\ge 0$, which shows that the SHC of $X$ is related to the solution to \eqref{eq:integro_diff}, that is,
	\begin{align*}
		Q^X_\Omega(t)=\int_\Omega u_\varphi(x,t)dx.
	\end{align*}
	In particular, when $\alpha=2$, $u_\varphi$ solves (in $L^2$-weak sense) the following Cauchy problem with Dirichlet boundary conditions:
	\begin{align*}
		\partial^\varphi_tu(x,t)&=\Delta_x u(x,t) \quad \mbox{in $\Omega\times (0,\infty)$} \\
		u(x,t)&= 0 \quad \mbox{for $(x,t)\in \partial\Omega\times (0,\infty) $} \\
		u(x,0) &= 1 \quad \mbox{for $x\in\Omega$}.
	\end{align*}
	\section{Preliminaries: Perimeter and regularity of the boundary}\label{sec:3} A function $f\in L^1(\R^d)$ has bounded variation if
	\begin{align*}
		\|D f\|:=\sup\left\{\int_{\R^d} f\sum_{i=1}^d D_i\varphi_i: \varphi_i\in C^\infty_c(\R^d), \sum_{i=1}^d \varphi^2_i\le 1\right\}<\infty.
	\end{align*}
	In this case, $\|Df\|$ is called the variation of $f$, and the space of all functions with bounded variation is denoted by $\BV(\R^d)$. A bounded open set $\Omega\subset\R^d$ is said to have \emph{finite perimeter} if $\mathbbm{1}_\Omega\in \BV(\R^d)$, and we write
	\begin{align*}
		\Per(\Omega)=\|D\mathbbm{1}_\Omega\|.
	\end{align*}
	When $f\in C^\infty_c(\R^d)$, due to integration by parts formula, its variation is given by
	\begin{align*}
		\|Df\|=\int_{\R^d}|\nabla f(x)|dx.
	\end{align*}
	In general, the mapping 
	\begin{align*}
		\mathrm{BV}(\R^d)\longmapsto [0,\infty) \\
		f\longmapsto \|Df\|
	\end{align*}
	is lower semicontinuous in $L^1(\R^d)$-topology, that is, if $\{f,f_n:n\ge 1\}\subset L^1(\R^d)\cap \mathrm{BV}(\R^d)$ is such that $f_n\to f$ in $L^1(\R^d)$, then
	\begin{align*}
		\liminf_{n\to\infty} \|Df_n\|\ge \|Df\|,
	\end{align*}
	see \cite[Proposition~3.6]{AmbrosioBook} for more details.
	
	Next, for $\kappa>0$, a function $f:\R^d\to\R$ is $C^{1,\kappa}$ if $f$ is $C^1$ and $\nabla f$ is uniformly H\"older continuous with exponent $\kappa$, that is, there exists $L>0$ such that 
	\begin{align*}
		|\nabla f(x)-\nabla f(y)|\le L|x-y|^\kappa
	\end{align*}
	for all $x,y\in\R^d$. The domain $\Omega\subset\R^d$ has  $C^{1,\kappa}$ boundary if for every $x\in\partial\Omega$, there exists a neighborhood $V_x$ and a $C^{1,\kappa}$ function $\gamma:\R^{d-1}\to\R$ such that 
	\begin{align*}
		\Omega\cap V_x=\{(y_1,\ldots, y_d): y_d>\gamma(y_1,\ldots, y_{d-1})\}\cap V_x.
	\end{align*}
	If $\Omega$ has $C^{1}$ boundary, from Stoke's theorem it follows that $\Per(\Omega)=\mathcal{H}^{d-1}(\partial\Omega)$, the $(d-1)$-dimensional Hausdorff measure of $\partial\Omega$.
	
	\section{The minimum functional}\label{sec:4}
	Consider a nonnegative measurable function $f$ on $\R^d$. For any $u\ge 0$, let us define
	\begin{align*}
		E^f_u = \{x: f(x)>u\}.
	\end{align*}
	We consider the following functional version of the spectral heat content: for any $f$ as above, we define
	\begin{align*}
		Q^X_f(t)=\int_0^\infty Q^X_{E^f_u}(t) du = \int_0^\infty \int_{\R^d} \P_x(\tau^X_{E^f_u}>t) dx du.
	\end{align*}
	\begin{lemma}
		For any nonnegative measurable function $f$ on $\R^d$ one has
		\begin{align*}
			Q^X_f(t)=\int_{\R^d}\E_x\left[\inf_{0\le s\le t} f(X_s)\right]dx.
		\end{align*}
		In particular, $Q^X_{\mathbbm{1}_\Omega}(t)=Q^X_\Omega(t)$ for any open set $\Omega$.
	\end{lemma}
	\begin{proof}
		We note that for any fixed $u\ge 0$, 
		\begin{align*}
			\left\{u< \inf_{0\le s\le t} f(X_s)\right\}\subseteq	\{\tau^{X}_{E^f_u}>t\} = \left\{\inf_{s\ge 0}\{s:f(X_s)\le u\}>t\right\} \subseteq\left\{u\le \inf_{0\le s\le t} f(X_s)\right\}
		\end{align*}
		Applying Fubini's theorem, 
		\begin{align*}
			Q^X_f(t)=\int_{\R^d} \int_0^\infty \P_x\left(\inf_{0\le s\le t} f(X_s)\ge u\right) du dx=\int_{\R^d}\mathbb{E}_x\left[\inf_{0\le s\le t} f(X_s)\right] dx.
		\end{align*}
		This completes the proof of the lemma.
	\end{proof}
	In what follows, we will study the asymptotic behavior of the quantity 
	\begin{align*}
		\mathcal{R}^X_f(t):=\int_{\R^d} f(x) dx - Q^X_f(t).
	\end{align*}
	as $t\downarrow 0$. We note that $\mathcal{R}^X_{\mathbbm{1}_\Omega}(t)=\Vol(\Omega)-Q^X_\Omega(t)$.
	\begin{theorem}\label{thm:symmetric_smooth}
		Let $X=(X_t)_{t\ge 0}$ be a translation invariant, isotropic stochastic process on $\R^d$ satisfying Assumption~\ref{assump1}. Then for any nonnegative $f\in C^{1,\kappa}_c(\R^d)$,
		\begin{align*}
			\lim_{t\downarrow 0} \frac{\mathcal{R}^X_f(t)}{\mu(t)} = \int_{\R^d} |\nabla f(x)| dx,
		\end{align*}
		where $\mu(t)$ is defined in \eqref{eq:mu_sigma}. In particular, the above limit holds for all compactly supported smooth functions.
	\end{theorem}
	We require the following lemmas to prove the above theorem.
	\begin{lemma}\label{lem:limits}
		For any $\varepsilon>0$,
		\begin{align}
			&\lim_{t\downarrow 0}\frac{1}{\mu(t)}\E_0\left[\overline{X}^{(1)}_t\wedge \varepsilon\right]=1 \label{eq:lim1}\\
			&\lim_{t\downarrow 0}\frac{1}{\mu(t)}\E_0\left[\overline{X}^{(1)}_t\mathbbm{1}\left\{\overline{X}^{(1)}_t\le \varepsilon\right\}\right]=1 \label{eq:lim0}\\
			&\lim_{t\downarrow 0}\frac{1}{\mu(t)}\E_0\left[\overline{X}^{(1)}_t\mathbbm{1}\left\{\sup_{0\le s\le t}|X_s|\le \varepsilon\right\}\right]=1 \label{eq:lim2} \\
			&\limsup_{t\downarrow 0}\frac{1}{\mu(t)}\E_0\left[\sup_{0\le s\le t} |X_s|\mathbbm{1}\left\{\sup_{0\le s\le t}|X_s|\le \varepsilon\right\}\right]\le 2d \label{eq:lim3}
		\end{align} 
	\end{lemma}
	\begin{proof}
		For any $\varepsilon>0$ we note that
		\begin{align*}
			\left|\E_0\left[\overline{X}^{(1)}_t\wedge \varepsilon\right]-\mu(t)\right|\le \max\{\varepsilon,1\}\P_0\left(\overline{X}^{(1)}_t>\varepsilon\wedge 1\right).
		\end{align*}
		Therefore, \eqref{eq:lim1} follows as $\lim_{t\downarrow 0}\P_0\left(\overline{X}^{(1)}_t>\varepsilon\wedge 1\right)/\mu(t)=0$ due to Assumption~\ref{assump1}. Next, for any $t>0$, we have
		\begin{align*}
			0&\le \E_0\left[\overline{X}^{(1)}_t\wedge\varepsilon\right]-\E_0\left[\overline{X}^{(1)}_t\mathbbm{1}\left\{\overline{X}^{(1)}_t\le \varepsilon\right\}\right] \\
			&= \varepsilon\P_0\left(\overline{X}^{(1)}_t>\varepsilon\right),
		\end{align*}
		and hence, \eqref{eq:lim0} follows from \eqref{eq:lim1} and Assumption~\ref{assump1}.
		To prove \eqref{eq:lim2}, we observe that
		\begin{align*}
			0&\le\E_0\left[\overline{X}^{(1)}_t\mathbbm{1}\{\overline{X}^{(1)}_t\le \varepsilon\}\right]-\E_0\left[\overline{X}^{(1)}_t\mathbbm{1}\{\sup_{0\le s\le t}|X_s|\le \varepsilon\}\right] \\
			&= \E_0\left[\overline{X}^{(1)}_t\mathbbm{1}\{\overline{X}^{(1)}_t\le \varepsilon, \sup_{0\le s\le t} |X_s|>\varepsilon\}\right] \\
			& \le \varepsilon \P_0\left(\sup_{0\le s\le t} |X_s|>\varepsilon\right)\le 2d\epsilon\P_0\left(\overline{X}^{(1)}_s>\epsilon/\sqrt{d}\right),
		\end{align*}
		where the last inequality follows from the isotropy of $X_t$.
		Again by Assumption~\ref{assump1} and \eqref{eq:lim0}, we conclude that 
		\begin{equation}\label{eq:sup_exp}
			\begin{aligned}
				&\ \ \ \ \lim_{t\downarrow 0}\frac{1}{\mu(t)}\E_0\left[\overline{X}^{(1)}_t\mathbbm{1}\{\sup_{0\le s\le t}|X_s|\le \varepsilon\}\right] \\
				&=\lim_{t\downarrow 0}\frac{1}{\mu(t)}\E_0\left[\overline{X}^{(1)}_t\mathbbm{1}\{\overline{X}^{(1)}_t\le \varepsilon\}\right]=1.
			\end{aligned}
		\end{equation}
		Finally, using the symmetry of the random variables $X^{(i)}_t$ for $i=1,\ldots, d$, we note that 
		\begin{align*}
			&\E_0\left[\sup_{0\le s\le t} |X^{(i)}_s|\mathbbm{1}\{\sup_{0\le s\le t} |X^{(i)}_s|\le \varepsilon\}\right]\le\int_0^\varepsilon\P_0\left(\sup_{0\le s\le t} |X^{(i)}_s|>u\right)du \\
			&\le 2\int_0^\varepsilon\P_0\left(\overline{X}^{(i)}_t>u\right)du = 2\E_0\left[\overline{X}^{(i)}_t\wedge\varepsilon\right].
		\end{align*}
		Since $|X_s|\le \sum_{i=1}^d |X^{(i)}_s|$, it follows that
		\begin{align*}
			\E_0\left[\sup_{0\le s\le t} |X_s|\mathbbm{1}\{\sup_{0\le s\le t}|X_s|\le \varepsilon\}\right]\le \sum_{i=1}^d\E_0\left[\sup_{0\le s\le t} |X^{(i)}_s|\mathbbm{1}\{\sup_{0\le s\le t}|X^{(i)}_s|\le \varepsilon\}\right].
		\end{align*}
		Therefore, \eqref{eq:lim3} follows from \eqref{eq:sup_exp} and \eqref{eq:lim2}.
	\end{proof}
	\begin{lemma}\label{lem:taylor_bound}
		Suppose that $f\in C^{1,\kappa}(\R^d)$ with $|\nabla f(x)-\nabla f(y)|\le L|x-y|^\kappa$ for all $x,y\in\R^d$, and $t>0$. Then,
		\begin{align*}
			\left|\inf_{0\le s\le t} f(X_s)-f(x)-\inf_{0\le s\le t} \langle X_s-x,\nabla f(x)\rangle\right|\le \frac{L}{1+\kappa}\sup_{0\le s\le t} |X_s-x|^{1+\kappa}
		\end{align*}
	\end{lemma}
	\begin{proof}
		Since $f\in C^{1}(\R^d)$, using Taylor expansion around $x$, we obtain,
		\begin{align}\label{eq:taylor1}
			f(X_s) = f(x) + \langle X_s-x, \nabla f(x)\rangle + \omega_f(x, X_s)
		\end{align}
		where $\omega_f(x,x')=\int_0^1 \langle(x'-x), (\nabla f(x+u(x'-x))-\nabla f(x))\rangle du$.
		Using the $\kappa$-H\"older continuity of $\nabla f$, it follows that
		\begin{align}\label{eq:taylor_bound}
			|\omega_f(x,x')|\le \frac{L}{1+\kappa}|x-x'|^{1+\kappa}.
		\end{align}
		This completes the proof of the lemma.
	\end{proof}
	We also require the following lemma that will be used throughout the article.
	\begin{lemma}\label{lem:isotropy}
		Let $X=(X_t)_{t\ge 0}$ be any translation invariant isotropic process in $\R^d$. Then for all $a,x\in\R^d$,
		\begin{align*}
		 \mbox{$\P_x$--law of $\left(\sup_{0\le s\le t}\langle X_s-x, a\rangle, \sup_{0\le s\le t}|X_s-x|\right)$} = \mbox{$\P_0$--law of $\left(|a| \overline{X}^{(1)}_t, \sup_{0\le s\le t}|X_s|\right)$}
		\end{align*}
	\end{lemma}
	\begin{proof}
		Since $X$ is translation invariant, without loss of generality we can assume that $x=0$. If $a=0$, the statement of the lemma is obvious due to isotropy of $X$. Let us assume that $a\neq 0$ and $U$ be an orthogonal matrix such that $Ue_1=a/|a|$, where $e_1=(1,0,\ldots, 0)\in\R^d$. Then by isotropy of $X$, 
		\begin{align*}
			&\ \ \  \mbox{$\P_0$--law of $\left(\sup_{0\le s\le t}\langle X_s, a\rangle, \sup_{0\le s\le t}|X_s|\right)$}\\
			&=\mbox{$\P_0$--law of $\left(\sup_{0\le s\le t}\langle UX_s, a\rangle, \sup_{0\le s\le t}|UX_s|\right)$} \\
			&=\mbox{$\P_0$--law of $\left(\sup_{0\le s\le t}\langle X_s, |a|e_1\rangle, \sup_{0\le s\le t}|X_s|\right)$}.
		\end{align*}
		This proves the lemma.
	\end{proof}
	\begin{proof}[Proof of Theorem~\ref{thm:symmetric_smooth}]
		Let us write $S_f=\supp(f)$ and let 
		\begin{align*}
			\mathcal U=\{x\in\R^d: \dist(x, S_f)<\eta\}
		\end{align*}
		for some $\eta>0$, where $\dist(x,S_f)=\inf\{|x-y|: y\in S_f\}$. Then $\overline{S_f}\subset \mathcal U$.
		We observe that 
		\begin{align*}
			f(x)=0=\E_x\left[\inf_{0\le s\le t} f(X_s)\right]
		\end{align*}
		whenever $x\notin \mathcal U$.
		As a result, using \eqref{eq:taylor1}, $\mathcal{R}^X_f(t)$ reduces to
		\begin{equation*}
			\mathcal{R}^X_f(t)=\mathcal{I}_1(t)+\mathcal{I}_2(t),
		\end{equation*}
		where
		\begin{equation}\label{eq:R_f}
		\begin{aligned}
			\mathcal{I}_1(t)&=\int_{\mathcal U}\mathbb{E}_x\left[-\inf_{0\le s\le t} \left(\langle X_s-x,\nabla f(x)\rangle + \omega_f(x,X_s)\right)\mathbbm{1}\{X_s\in \mathcal U, 0\le s\le t\}\right] dx, \\
			\mathcal{I}_2(t)&=\int_{\mathcal U} f(x)\P_x(\exists s\in [0,t]: X_s\notin\mathcal U)dx.
		\end{aligned}
		\end{equation}
		Let us first show that 
		\begin{align}\label{eq:lim_II}
			\lim_{t\to 0}\frac{\mathcal I_2(t)}{\mu(t)}=0.
		\end{align}
		Indeed, as $f$ vanishes outside its support $S_f$, the integral defining $\mathcal{I}_2(t)$ reduces to
		\begin{align*}
			\mathcal{I}_2(t)&=\int_{S_f} f(x)\P_x(\exists s\in [0,t]: X_s\notin \mathcal U) dx \\
			&\le \|f\|_\infty\int_{S_f} \P_x(\exists s\in [0,t]: |X_s-x|\ge\eta)dx \\
			&=\|f\|_\infty\Vol(S_f)\P_0\left(\sup_{0\le s\le t} |X_s|\ge\eta\right),
		\end{align*}
		where the last identity follows from the translation invariance property of $(X_t)_{t\ge 0}$. Using the union bound and isotropy of $(X_t)_{t\ge 0}$ it follows that
		\begin{align*}
			\P_0\left(\sup_{0\le s\le t} |X_s|\ge\eta\right)\le 2d \,\P_0\left(\overline{X}^{(1)}_t\ge \eta/\sqrt{d}\right).
		\end{align*}
		Hence, \eqref{eq:lim_II} holds due to Assumption~\ref{assump1}.

		Let us now deal with the first term $\mathcal{I}_1(t)$ in \eqref{eq:R_f}. Assuming that $\mathcal U\subset B_{R}(0)$ for some $R>0$, \eqref{eq:R_f} implies
		\begin{equation}\label{eq:R_f_bound}
			\begin{aligned}
				&\left|\mathcal{I}_1(t)-\int_{\mathcal U}\E_x\left[\sup_{0\le s\le t}-\langle X_s-x,\nabla f(x)\rangle \mathbbm{1}\{X_s\in \mathcal U, 0\le s\le t\} \right]dx\right| \\
				&\le \int_{\mathcal U}\E_x\left[\sup_{0\le s\le t}|\omega_f(x, X_s)|\mathbbm{1}\left\{\sup_{0\le s\le t}|X_s-x|\le 2R\right\}\right]dx \\
				& \le \frac{L}{1+\kappa}\int_{\mathcal U}\E_x\left[\sup_{0\le s\le t}|X_s-x|^{1+\kappa}\mathbbm{1}\left\{\sup_{0\le s\le t}|X_s-x|\le 2R\right\}\right]dx,
			\end{aligned}
		\end{equation}
		where we used the bound in \eqref{eq:taylor_bound}. For $\varepsilon>0$, the translation invariance and isotropy of $X$ leads to 
		\begin{align*}
			& \ \ \ \ \E_x\left[\sup_{0\le s\le t}|X_s-x|^{1+\kappa}\mathbbm{1}\left\{\sup_{0\le s\le t}|X_s-x|\le 2R\right\}\right] \\
			&\le \varepsilon^\kappa \E_0\left[\sup_{0\le s\le t}|X_s|\mathbbm{1}\{\sup_{0\le s\le t}|X_s|\le\varepsilon\}\right]+2^{1+\kappa}R^{1+\kappa}\P_0\left(\sup_{0\le s\le t} |X_s|>\varepsilon\right).
		\end{align*}
		This shows with Assumption~\ref{assump1} and \eqref{eq:lim3} that for any $\varepsilon>0$,
		\begin{align*}
			\limsup_{t\downarrow 0}\frac{1}{\mu(t)}\int_{\mathcal U}\E_x\left[\sup_{0\le s\le t}|X_s-x|^{1+\kappa}\mathbbm{1}\left\{\sup_{0\le s\le t}|X_s-x|\le 2R\right\}\right] dx 
			\le 2d\varepsilon^\kappa \Vol(\mathcal U).
		\end{align*}
		Letting $\varepsilon\to 0$, the above inequality combined with \eqref{eq:R_f_bound} yields
		\begin{align}\label{eq:lim_remainder}
			\lim_{t\downarrow 0}\frac{1}{\mu(t)}\left|\mathcal{I}_1(t)-\int_{\mathcal U}\E_x\left[\sup_{0\le s\le t}-\langle X_s-x,\nabla f(x)\rangle \mathbbm{1}\{X_s\in\mathcal U, 0\le s\le t\} \right]dx\right|=0.
		\end{align}
		Let us denote 
		\begin{align*}
			\delta(x)=\inf\{|x-y|: y\notin \mathcal U\}.
		\end{align*}
		Then, any $x\in\mathcal U$ and $t\ge 0$ we have 
		\begin{equation}\label{eq:E_S_f}
			\begin{aligned}
				&\ \ \ \sup_{0\le s\le t}-\langle X_s-x,\nabla f(x)\rangle \mathbbm{1}\{\sup_{0\le s\le t}|X_s-x|\le \delta(x)\} \\
				& \le \sup_{0\le s\le t}-\langle X_s-x,\nabla f(x)\rangle \mathbbm{1}\{X_s\in\mathcal U, 0\le s\le t\} \\
				&\le \sup_{0\le s\le t}-\langle X_s-x,\nabla f(x)\rangle\mathbbm{1}\{\sup_{0\le s\le t}|X_s-x|\le 2R\}
			\end{aligned}
		\end{equation}
		Recalling the isotropy and translation invariance of $X$, Lemma~\ref{lem:isotropy} implies that for any $r>0$
		\begin{align*}
			&\ \ \ \ \mathbb{E}_x\left[\sup_{0\le s\le t}- \langle X_s-x,\nabla f(x)\rangle\mathbbm{1}\{\sup_{0\le s\le t}|X_s-x|\le r\}\right] \\
			&= \mathbb{E}_0\left[\overline{X}^{(1)}_t\mathbbm{1}\{\sup_{0\le s\le t}|X_s|\le r\}\right]|\nabla f(x)|=:\mu(t,r)|\nabla f(x)|.
		\end{align*}
		As a result, \eqref{eq:E_S_f} implies the following upper and lower bound:
		\begin{equation}\label{eq:E_f_2}
			\begin{aligned}
				& \mu(t, \delta(x))|\nabla f(x)| \\
				&\le \E_x\left[\sup_{0\le s\le t}-\langle X_s-x,\nabla f(x)\rangle \mathbbm{1}\{X_s, x\in\mathcal U, 0\le s\le t\}\right] \\ 
				& \le \mu(t,2R)|\nabla f(x)|.
			\end{aligned}
		\end{equation}
		Combining \eqref{eq:lim_remainder} and \eqref{eq:E_f_2}, we obtain
		\begin{equation}\label{eq:limsup}
			\begin{aligned}
				\limsup_{t\downarrow 0}\frac{\mathcal{I}_1(t)}{\mu(t)}&\le \limsup_{t\downarrow 0}\frac{\mu(t,2R)}{\mu(t)}\int_{\R^d}|\nabla f(x)|dx \\
				&=\int_{\R^d}|\nabla f(x)|dx,
			\end{aligned}
		\end{equation}
		where the last equality follows from \eqref{eq:lim2} in Lemma~\ref{lem:limits}.
		On the other hand, using \eqref{eq:lim_remainder} and \eqref{eq:E_f_2} once again we get
		\begin{align*}
			\liminf_{t\downarrow 0}\frac{\mathcal{I}_1(t)}{\mu(t)}\ge \liminf_{t\downarrow 0}\int_{\mathcal U}|\nabla f(x)|\frac{\mu(t,\delta(x))}{\mu(t)}dx.
		\end{align*}
		Since by \eqref{eq:lim2} in Lemma~\ref{lem:limits}, $\frac{\mu(t, \delta(x))}{\mu(t)}\to 1$ as $t\downarrow 0$, applying Fatou's lemma we conclude that
		\begin{align}\label{eq:liminf}
			\liminf_{t\downarrow 0}\frac{\mathcal{I}_1(t)}{\mu(t)}\ge \int_{\mathcal U}|\nabla f(x)|dx=\int_{\R^d} |\nabla f(x)|dx.
		\end{align}
		Combining \eqref{eq:limsup} and \eqref{eq:liminf}, we conclude the proof of the theorem.
	\end{proof}
	
	\section{Proof of the lower bound}\label{sec:lower_bound}
	\begin{proposition}\label{prop:liminf}
		Let $\Omega$ be a bounded, open set in $\R^d$ with finite perimeter. Then, 
		\begin{align*}
			\liminf_{t\downarrow 0}\frac{\Vol(\Omega)-Q^X_\Omega(t)}{\mu(t)}\ge \Per(\Omega).
		\end{align*}
		\begin{remark}
			When $\Omega$ has $C^{1,1}$ boundary, the above proposition can be proved using a half-space argument similar to the proof of  \cite[Lemma~3.1]{ParkSong2022}. However, the argument in the aforementioned paper requires the $R$-smooth boundary condition, which does not hold when $\partial\Omega$ is not $C^{1,1}$. 
		\end{remark}
	\end{proposition}
	The above result is proved using approximation argument using mollifiers. 
	Fix a collection of mollifiers $(\rho_\varepsilon)_{\varepsilon>0}\subset C^\infty_c(\R^d)$ such that $\int_{\R^d}\rho_\varepsilon(x)dx=1$ and for any $\varepsilon>0$ and $f\in L^1(\R^d)$, define the convolution
	\begin{align*}
		f_\varepsilon = f\ast\rho_\epsilon.
	\end{align*}
	\begin{lemma}\label{lem:mollifier}
		For any $f\in L^1(\R^d)$ with $f\ge 0$ and $\epsilon>0$, we have
		\begin{align*}
			\mathcal{R}^X_{f_\epsilon}(t)\le \mathcal{R}^X_f(t)
		\end{align*}
		for all $t\ge 0$,
		where $f_\epsilon = f\ast \rho_\epsilon$.
	\end{lemma}
	\begin{proof}
		We note that $\int_{\R^d} f_\epsilon(x)dx = \int_{\R^d} f(x)dx$ for all $\epsilon>0$. Also, by translation invariance of $X_t$ and Fubini's theorem we get 
		\begin{align*}
			\E_x\left[\inf_{0\le s\le t} f_\epsilon(X_s)\right] &= \E_0\left[\inf_{0\le s\le t} f_\epsilon(X_s+x)\right] \\
			&=\E_0\left[\inf_{0\le s\le t} \int_{\R^d}f(X_s+x-y)\rho_\epsilon(y)dy\right] \\
			& \ge \E_0\left[\int_{\R^d} \inf_{0\le s\le t} f(X_s+x-y)\rho_\epsilon(y) dy\right] \\
			& = \int_{\R^d}\E_{x-y}\left[\inf_{0\le s\le t} f(X_s)\right]\rho_\epsilon(y)dy.
		\end{align*}
		This shows that for all $\epsilon>0$,
		\begin{align*}
			\int_{\R^d} \E_{x}\left[\inf_{0\le s\le t} f_\epsilon(X_s)\right]dx\ge \int_{\R^d} \E_{x}\left[\inf_{0\le s\le t} f(X_s)\right]dx.
		\end{align*}
		This completes the proof of the lemma.
	\end{proof}
	\begin{proof}[Proof of Proposition~\ref{prop:liminf}] Let us consider $f=\mathbbm{1}_\Omega$. By Lemma~\ref{lem:mollifier} and Theorem~\ref{thm:symmetric_smooth}, for any $\epsilon>0$,
		\begin{align*}
			\liminf_{t\downarrow 0}\frac{\Vol(\Omega)-Q^X_\Omega(t)}{\mu(t)}\ge \liminf_{t\downarrow 0}\frac{\mathcal{R}^X_{f_\epsilon}(t)}{\mu(t)}=\int_{\R^d} |\nabla f_\epsilon(x)|dx.
		\end{align*}
		As $\lim_{\epsilon\to 0}\|f_\epsilon -\mathbbm{1}_\Omega\|_{L^1}=0$, by lower semicontinuity of the variation it follows that
		\begin{align*}
			\liminf_{t\downarrow 0}\frac{\Vol(\Omega)-Q^X_\Omega(t)}{\mu(t)}\ge \liminf_{\epsilon\to 0} \int_{\R^d} |\nabla f_\epsilon(x)|dx\ge \|D\mathbbm{1}_\Omega\|=\Per(\Omega).
		\end{align*}
		This completes the proof of the proposition.
	\end{proof}
	\section{Proof of Theorem~\ref{thm:main}}\label{sec:6} We prove a stronger version of Theorem~\ref{thm:main}.
	\begin{theorem}\label{thm:main2}
		Let $\phi:\R^d\to\R$ be a $C^{1,\kappa}$ function with $0<\kappa \ \le 1$ and assume that 
		\begin{enumerate}
			\item $\phi^{-1}([0,\infty))$ is compact, and
			\item $|\nabla \phi|$ is bounded away from $0$ in a neighborhood of $\phi^{-1}(0)$.
		\end{enumerate}
		Let $\Omega = \phi^{-1}((0,\infty))$. Then for any translation invariant isotropic process $X=(X_t)_{t\ge 0}$ satisfying Assumption~\ref{assump1} we have
		\begin{align*}
			\Per(\Omega)\le \liminf_{t\downarrow 0}\frac{\Vol(\Omega)-Q^X_\Omega(t)}{\mu(t)}\le \limsup_{t\downarrow 0}\frac{\Vol(\Omega)-Q^X_\Omega(t)}{\mu(t)}\le \frac{\sup_{\partial\Omega}|\nabla\phi|}{\inf_{\partial\Omega}|\nabla \phi|}\Per(\Omega).
		\end{align*}
	\end{theorem}
		\begin{remark}\label{rem:boundary}
		The conditions in Theorem~\ref{thm:main2} ensure that for sufficiently small $\varepsilon>0$, one has 
		\begin{align}\label{eq:boundary_equality}
			\phi^{-1}(r)=\partial \{x: \phi(x)>r\} \quad \mbox{for all $0\le r\le \epsilon$}.
		\end{align}
		Indeed, one can find an open neighborhood $V$ of $\phi^{-1}(0)$ such that $\{x: 0\le \phi(x)\le \epsilon\}\subset V$ for sufficiently small $\epsilon>0$, and $|\nabla \phi|$ is bounded away from $0$ on $V$. In particular, for each $0\le r\le \epsilon$, 
		\begin{align*}
			|\nabla \phi|\neq 0 \quad \mbox{on} \ \ \phi^{-1}(r).
		\end{align*}
		For any $x_0\in \phi^{-1}(r)$ and $h\in \R$, writing $y=\nabla\phi(x_0)/|\nabla \phi(x_0)|$ and using $C^1$-smoothness of $\phi$ we have 
		\begin{align*}
			\phi(x_0+hy)=\phi(x_0)+h|\nabla \phi(x_0)|+\mathrm{o}(h)=r +h|\nabla \phi(x_0)|+\mathrm{o}(h)\quad \mbox{as $h\to 0$},
		\end{align*}
		which shows that, for sufficiently small $h>0$,
		\begin{align*}
			\phi(x_0+hy)>r, \quad \mbox{and}  \quad \phi(x_0-hy)<r.
		\end{align*} 
		Hence, $x_0\in \partial \{x:\phi(x)>r\}$ which implies that $\phi^{-1}(r)\subseteq \partial \{x:\phi(x)>r\}$. On the other hand, due to continuity of $\phi$ we have $\partial \{x:\phi(x)>r\}\subseteq \phi^{-1}(r)$, which proves \eqref{eq:boundary_equality}.
	\end{remark}
	Theorem~\ref{thm:main2} is proved in several intermediate steps. First, we note the following result about continuity of the perimeters of level sets.
	\begin{lemma}\label{lem:perimeter_phi}
		Let $\phi:\R^d\to\R$ be a $C^{1}$ function such that 
		\begin{enumerate}
			\item $\phi^{-1}([0,\infty))$ is compact, and
			\item $|\nabla\phi|$ is bounded away from $0$ in a neighborhood of $\phi^{-1}(0)$. 
		\end{enumerate}
		Then, 
		\begin{align*}
			\lim_{r\downarrow 0}\mathcal{H}^{d-1}(\phi^{-1}(r))=\mathcal{H}^{d-1}(\phi^{-1}(0)).
		\end{align*}
	\end{lemma}
	 We refer the reader to the Mathoverflow post \cite{mathoverflow} and answer by Mathoverflow user \href{https://mathoverflow.net/users/56624/kostya-i}{Kostya\_I}, where the lemma has been proved for smooth functions $\phi$, but the same argument works for $C^1$ functions as the proof only relies on inverse function theorem. For the sake of completeness, we provide the proof below with some minor modifications.
	\begin{proof} 
		We claim that every point $x\in\phi^{-1}(0)$ has a neighborhood $V_x$ and there exists $r_x>0$ such that one can find a diffeomorphism $\varphi_x: Q_{r_x}\to V_x$ where
		\begin{align*}
			Q_{r_x}=\{|y_i|\leq r_x,\;1\leq i \leq d\},
		\end{align*}
		and $\varphi_x$ satisfies the following:
		\begin{enumerate}
			\item $\phi(\varphi_x(y))=y_1|\nabla \phi(x)| $, 
			\item $\varphi_x(0)=x$,
			\item $D\varphi_x(0)$ is a rotation.
		\end{enumerate}
	 Indeed, without loss of generality we may assume that $x=0$, and, by rotating the coordinates, that $\nabla \phi (0)/|\nabla \phi (0)|=e_1,$ the first basis vector. Consider the function $g:\mathbb{R}^d\to \mathbb{R}^d$ defined by 
	 \begin{align*}
	 	g(y)=(\phi(y)/|\nabla \phi (0)|,y_2,\dots,y_d).
	 \end{align*}
	  Then $Dg(0)$ is the identity matrix. Choose $\varphi_x=g^{-1}$ provided by the inverse function theorem. Let $z=g(y)$. Then, 
	  \begin{align*}
	  	z_1=\frac{\phi(y)}{|\nabla \phi(0)|},
	  \end{align*}
	  which shows that
	  \begin{align*}
	  	z_1=\frac{\phi(\varphi_x(z))}{|\nabla \phi(0)|}.
	  \end{align*}
	    Given $\varepsilon>0$, by choosing smaller $r_x$ if necessary, we can ensure that the restriction of $\varphi_x$ to the leaves $\{y_1=r\}$ distorts the $(d-1)$-area by no more than $1+\varepsilon$, that is, in $Q_{r_x}$,
		\begin{align}\label{eq:jacobian}
			1-\varepsilon\le \left(\det_{2\leq i,j\leq d}\left(\partial_{y_i}\varphi_x\cdot\partial_{y_j}\varphi_x\right)\right)^\frac12\leq 1+\varepsilon.
		\end{align}
		By compactness of $\phi^{-1}([0,\infty))$, we can choose a finite cover $(V_{x_i})_{i=1}^n$ of $\phi^{-1}(0)$ such that there exists $\epsilon_0>0$ so that
		\begin{align*}
			\{x: 0\le \phi(x)\le \epsilon_0\}\subset \cup_{i=1}^n V_{x_i}.
		\end{align*}
		We note that by construction of $\varphi_x$ above, for each $1\le i\le n$, we have that for all $0\le r< \min\{r_{x_i} |\nabla \phi(x_i)|, \epsilon_0\}$,
		\begin{align*}
			\phi^{-1}(r)\cap V_{x_i}=\varphi_{x_i}\left(\left\{y\in\R^d: y_1=\frac{r}{|\nabla \phi(x_i)|}, \ |y_j|\le r_{x_i} \ \forall j\ge 2\right\}\right).
		\end{align*} 
		Let $f_1,\dots,f_n$ be a partition of unity  subordinate to $(V_{x_i})_{i=1}^n$. 
		Writing \[r_0=\min\{\epsilon_0, r_{x_i}|\nabla \phi(x_i)|: i=1,\ldots, n\},\] for all $0\le r<r_0$ we have
		\begin{equation*}
			\mathcal{H}^{d-1}(\phi^{-1}(r))=\sum_{i=1}^n\int_{\phi^{-1}(r)} f_i(w) \mathcal{H}^{d-1}(dw).
		\end{equation*}
	For each integral above, by changing the variable $w=\varphi_{x_i}(y)$ one gets
	\begin{align*}
		\int_{\phi^{-1}(r)} f_i(w) \mathcal{H}^{d-1}(dw)=\int_{Q^r_i} f_i\circ \varphi_{x_i}(y) \left(\det_{2\leq i,j\leq d}\left(\partial_{y_i}\varphi_{x_i}(y)\cdot\partial_{y_j}\varphi_{x_i}(y)\right)\right)^\frac12 \, dy_2\cdots dy_d.
	\end{align*}
	Hence,
	 \eqref{eq:jacobian} yields that for all $0\le r<r_0$, 
		\begin{equation}\label{eq:changed_var}
			(1-\varepsilon)\sum_{i=1}^n\int_{Q^r_i} f_i\circ\varphi_i\, dy_2\ldots dy_d\leq \mathcal{H}^{d-1}(\phi^{-1}(r))\leq (1+\varepsilon)\sum_{i=1}^n\int_{Q^r_i} f_i\circ\varphi_i\, dy_2\ldots dy_d,
		\end{equation}
		where $Q^r_i=Q_{x_i}\cap\{y_1=r/|\nabla \phi(x_i)|\}$.
		Since $f_i, \varphi_{x_i}$ are continuous functions, we have
		\begin{align}\label{eq:integral_lim}
			\lim_{r\downarrow 0} \int_{Q^r_i} f_i\circ\varphi_{x_i}\, dy_2\ldots dy_d=\int_{Q^0_i} f_i\circ\varphi_{x_i}\, dy_2\ldots dy_d.
		\end{align}
		Hence, combining \eqref{eq:changed_var} and \eqref{eq:integral_lim} we obtain that for any $\epsilon>0$,
		\begin{align*}
			\frac{1-\varepsilon}{1+\epsilon}\mathcal{H}^{d-1}(\phi^{-1}(0))&\leq\liminf_{r\downarrow 0} \mathcal{H}^{d-1}(\phi^{-1}(r)), \\
			\limsup_{r\downarrow 0} \mathcal{H}^{d-1}(\phi^{-1}(r))&\leq \frac{1+\varepsilon}{1-\epsilon}\mathcal{H}^{d-1}(\phi^{-1}(0)).
		\end{align*}
		Since $\varepsilon>0$ is arbitrary, the proof of the lemma follows.
	\end{proof}
	
	\begin{notation}
		For the rest of the article, we denote  
		\begin{equation*}
			\begin{aligned}
				\Omega'_r=\{x:0<\phi(x)<r\}, \quad \Omega_r=\{x:\phi(x)>r\}, \quad r>0.
			\end{aligned}
		\end{equation*}
		Then due to Remark~\ref{rem:boundary}, $\partial\Omega_r=\phi^{-1}(r)$ for sufficiently small nonnegative values of $r$.
	\end{notation}
	
	\begin{lemma}\label{lem:away}
		For any $\varepsilon>0$,
		\begin{align*}
			\lim_{t\downarrow 0}\frac{1}{\mu(t)}\int_{\Omega_\epsilon}\P_x\left(\inf_{0\le s\le t} \phi(X_s)\le 0\right)dx =0
		\end{align*}
	\end{lemma}
	\begin{proof}
		
		Invoking Lemma~\ref{lem:taylor_bound} we have
		\begin{align*}
			\inf_{0\le s\le t} \phi(X_s)\ge \phi(x)+\inf_{0\le s\le t}\langle X_s-x,\nabla \phi(x)\rangle-\frac{L}{1+\kappa}\sup_{0\le s\le t}|X_s-x|^{1+\kappa}.
		\end{align*}
		Therefore, using probability union bound we obtain 
		\begin{align*}
			& \ \ \ \ \P_x\left(\inf_{0\le s\le t} \phi(X_s)\le 0\right) \\
			&\le \P_x\left(\sup_{0\le s\le t} -\langle X_s-x,\nabla \phi(x)\rangle+\frac{L}{1+\kappa}\sup_{0\le s\le t}|X_s-x|^{1+\kappa}>\phi(x)\right) \\
			&\le \P_x\left(\sup_{0\le s\le t} |X_s-x|>\frac{\phi(x)}{2M}\right)+\P_x\left(\sup_{0\le s\le t}|X_s-x|>\left(\frac{\phi(x)}{2L'}\right)^{1/1+\kappa}\right),
		\end{align*}
		where 
		\begin{align*}
			M=\sup_{x\in\overline{\Omega}_\epsilon}|\nabla \phi(x)|<\infty, \quad \mbox{and} \quad  L'=L/(1+\kappa),
		\end{align*}
		as $\overline{\Omega}_\epsilon$ is compact.
		 As a result, the translation invariance property of $(X_t)_{t\ge 0}$ and Assumption~\ref{assump1} imply that
		\begin{align*}
			&\ \ \ \ \limsup\limits_{t\downarrow 0}\frac{1}{\mu(t)}\int_{\Omega_\epsilon}\P_x\left(\inf_{0\le s\le t} \phi(X_s)\le 0\right)dx \\
			&\le \limsup\limits_{t\downarrow 0}\frac{1}{\mu(t)}\left(\P_0\left(\sup_{0\le s\le t}|X_s|>\frac{\varepsilon}{2M}\right)+\P_0\left(\sup_{0\le s\le t} |X_s|>\left(\frac{\varepsilon}{2L'}\right)^{1/1+\kappa}\right)\right)\Vol(\Omega) \\
			& = 0.
		\end{align*}
		This completes the proof of the lemma.
	\end{proof}
	
	\begin{lemma}\label{lem:E_phi_0_e}
		Let $\eta\in (0,1)$ and $\epsilon>0$ be sufficiently small such that
		\begin{align*}
			\inf_{x\in \Omega'_\epsilon} |\nabla \phi(x)|>0.
		\end{align*}
		Then,
		\begin{align*}
			&\limsup_{t\downarrow 0}\frac{1}{\mu(t)}\int_{\Omega'_\epsilon}\P_x\left(\inf_{0\le s\le t}\phi(X_s)\le 0\right)dx \\
			&\le \sup_{0\le v\le \epsilon}\mathcal{H}^{d-1}(\partial\Omega_v)\frac{M_\epsilon}{\eta m_\epsilon}+ 2d\sup_{0\le v\le \epsilon}\mathcal{H}^{d-1}(\partial\Omega_v)\frac{(1-\eta)^{-\frac{1}{\kappa+1}}}{m_\epsilon} \epsilon^{\frac{\kappa}{\kappa+1}},
		\end{align*} 
		where $M_\epsilon=\sup_{x\in \Omega'_\epsilon} |\nabla \phi(x)|$ and $m_\epsilon=\inf_{x\in \Omega'_\epsilon} |\nabla \phi(x)|$.
	\end{lemma}
	\begin{proof}
		Fix any $\eta\in (0,1)$. For any $x\in \Omega'_\epsilon$, using Lemma~\ref{lem:taylor_bound} again we get
		\begin{align}
			&\ \ \ \ \P_x\left(\inf_{0\le s\le t}\phi(X_s)\le 0\right) \nonumber\\
			&\le \P_x\left(\sup_{0\le s\le t}-\langle X_s-x,\nabla \phi(x)\rangle>\eta\phi(x)\right)+\P_x\left(\sup_{0\le s\le t}|X_s-x|^{1+\kappa}>(1-\eta)\phi(x)\right). \label{eq:0_e}
		\end{align}
		Using Lemma~\ref{lem:isotropy} and the coarea formula, integral of the first term on $\Omega'_\epsilon$ reduces to
		\begin{equation}\label{eq:co-area1}
			\begin{aligned}
				&\ \ \ \ \int_{\Omega'_\epsilon}\P_x\left(\sup_{0\le s\le t}-\langle X_s-x,\nabla \phi(x)\rangle>\eta\phi(x)\right)dx \\
				&=\int_0^\epsilon\int_{\partial\Omega_v}\P_0\left(|\nabla \phi(x)|\overline{X}^{(1)}_t>\eta v\right)\frac{\mathcal{H}^{d-1}(dx)}{|\nabla \phi(x)|} dv.
			\end{aligned}
		\end{equation}
		Then, \eqref{eq:co-area1} yields
		\begin{align*}
			&\ \ \ \ \int_{\Omega'_\epsilon}\P_x\left(\sup_{0\le s\le t}-\langle X_s-x,\nabla \phi(x)\rangle>\eta\phi(x)\right)dx \\
			&\le \sup_{0\le v\le \epsilon}\mathcal{H}^{d-1}(\partial\Omega_v)\frac{1}{m_\epsilon}\int_0^\epsilon \P_0\left(M_\epsilon \overline{X}^{(1)}_t>\eta v \right)dv \\
			&= \sup_{0\le v\le \epsilon}\mathcal{H}^{d-1}(\partial\Omega_v)\frac{M_\epsilon}{\eta m_\epsilon}\E_0\left(\overline{X}^{(1)}_t\wedge \frac{\eta\epsilon}{M_\epsilon}\right)
		\end{align*} 
		By \eqref{eq:lim1} in Lemma~\ref{lem:limits}, we get
		\begin{equation}\label{eq:limsup_3}
			\begin{aligned}
				&\limsup_{t\downarrow 0}\frac{1}{\mu(t)}\int_{\Omega'_\epsilon}\P_x\left(\sup_{0\le s\le t}-\langle X_s-x,\nabla \phi(x)\rangle>\eta\phi(x)\right)dx \\
				&\le \sup_{0\le v\le \epsilon}\mathcal{H}^{d-1}(\partial\Omega_v)\frac{M_\epsilon}{\eta m_\epsilon}
			\end{aligned}
		\end{equation}
		To estimate the integral of the second term in \eqref{eq:0_e}, we again use coarea formula as follows:
		\begin{align}
			&\int_{\Omega'_\epsilon} \P_x\left(\sup_{0\le s\le t}|X_s-x|^{1+\kappa}>(1-\eta)\phi(x)\right) dx \nonumber \\
			&=\int_{\Omega'_\epsilon} \P_0\left(\sup_{0\le s\le t}|X_s|^{1+\kappa}>(1-\eta)\phi(x)\right)dx \nonumber \\
			&=\int_0^\epsilon\int_{\partial\Omega_v} \frac{\P_0\left(\sup_{0\le s\le t}|X_s|^{1+\kappa}>(1-\eta)v\right)}{|\nabla\phi(x)|}\mathcal{H}^{d-1}(dx) dv \nonumber\\
			&\le \sup_{0\le v\le \epsilon}\mathcal{H}^{d-1}(\partial\Omega_v)\frac{1}{(1-\eta)m_\epsilon}\E_0\left[\sup_{0\le s\le t}|X_s|^{1+\kappa}\wedge (1-\eta)\epsilon\right] \nonumber \\
			& \le 2d \sup_{0\le v\le \epsilon}\mathcal{H}^{d-1}(\partial\Omega_v)\frac{(1-\eta)^{-\frac{1}{\kappa+1}}}{m_\epsilon} \mu(t)\epsilon^{\frac{\kappa}{\kappa+1}}, \label{eq:0_e_2}
		\end{align}
		where in the last inequality we used 
		\begin{align*}
			&\E_0\left[\sup_{0\le s\le t}|X_s|^{1+\kappa}\wedge (1-\eta)\epsilon\right]=\E_0\left[\left(\sup_{0\le s\le t}|X_s|\wedge (1-\eta)^{\frac{1}{1+\kappa}}\epsilon^{\frac{1}{1+\kappa}}\right)^{\kappa+1}\right] \\
			&\le ((1-\eta)\epsilon)^{\frac{\kappa}{\kappa+1}}\E_0\left[\sup_{0\le s\le t}|X_s|\wedge (1-\eta)^{\frac{1}{1+\kappa}}\epsilon^{\frac{1}{1+\kappa}}\right]\\
			&\le  2d\mu(t)((1-\eta)\epsilon)^{\frac{\kappa}{\kappa+1}}
		\end{align*}
		for any $\epsilon\in (0,1)$. We note that the last inequality holds true since $X_t$ is isotropic, $(1-\eta)\varepsilon<1$, and therefore,
		\begin{align*}
			&\E_0\left[\sup_{0\le s\le t}|X_s|\wedge (1-\eta)^{\frac{1}{1+\kappa}}\epsilon^{\frac{1}{1+\kappa}}\right]  \le \E_0\left[\sup_{0\le s\le t} |X_s|\wedge 1\right] \\
			& \le d \, \E_0\left[\sup_{0\le s\le t} |X^{(1)}_s|\wedge 1\right] 
			 \le 2d \, \E_0\left[\sup_{0\le s\le t} X^{(1)}_s\wedge 1\right].
		\end{align*}
	
		 Since $m_\epsilon$ is bounded away from $0$ as $\epsilon\downarrow 0$, \eqref{eq:0_e_2} implies
		\begin{equation}\label{eq:limsup_4}
			\begin{aligned}
				&\limsup_{t\downarrow 0}\frac{1}{\mu(t)}\int_{\Omega'_\epsilon} \P_x\left(\sup_{0\le s\le t}|X_s-x|^{1+\kappa}>(1-\eta)\phi(x)\right) dx \\
				&\le 2d\sup_{0\le v\le \epsilon}\mathcal{H}^{d-1}(\partial\Omega_v)\frac{(1-\eta)^{-\frac{1}{\kappa+1}}}{m_\epsilon} \epsilon^{\frac{\kappa}{\kappa+1}}.
			\end{aligned}
		\end{equation}
		Hence, the proof of the lemma follows after combining \eqref{eq:limsup_3} and \eqref{eq:limsup_4}.
	\end{proof}
	\begin{proof}[Proof of Theorem~\ref{thm:main2}] By Proposition~\ref{prop:liminf}, it suffices to prove that 
		\begin{align*}
			\limsup_{t\downarrow 0}\frac{\Vol(\Omega)-Q^X_\Omega(t)}{\mu(t)}\le \frac{\sup_{\partial\Omega}|\nabla\phi|}{\inf_{\partial\Omega}|\nabla \phi|}\Per(\Omega).
		\end{align*}
		For any $\epsilon>0$ we have
		\begin{equation}\label{eq:limsup_5}
			\begin{aligned}
				& \ \ \ \ \limsup_{t\downarrow 0}\frac{\Vol(\Omega)-Q^X_\Omega(t)}{\mu(t)}=\frac{1}{\mu(t)}\int_{\Omega}\P_x(\tau^X_\Omega\le t) dx \\
				&\le \limsup_{t\downarrow 0}\frac{1}{\mu(t)}\int_{\Omega_\epsilon}\P_x\left(\inf_{0\le s\le t}\phi(X_s)\le 0\right) dx  \\ &+\limsup_{t\downarrow 0}\frac{1}{\mu(t)}\int_{\Omega'_\epsilon}\P_x\left(\inf_{0\le s\le t}\phi(X_s)\le 0\right)dx.
			\end{aligned}
		\end{equation}
		By Lemma~\ref{lem:away}, the first term on the right hand side of \eqref{eq:limsup_5} is zero. Recalling Lemma~\ref{lem:E_phi_0_e}, we get
		\begin{align*}
			& \ \ \ \ \limsup_{t\downarrow 0}\frac{\Vol(\Omega)-Q^X_\Omega(t)}{\mu(t)} \\
			&\le \sup_{0\le v\le \epsilon}\mathcal{H}^{d-1}(\partial\Omega_v)\frac{M_\epsilon}{\eta m_\epsilon}+ 2d\sup_{0\le v\le \epsilon}\mathcal{H}^{d-1}(\partial\Omega_v)\frac{(1-\eta)^{-\frac{\kappa}{\kappa+1}}}{m_\epsilon} \epsilon^{\frac{\kappa}{\kappa+1}}
		\end{align*}
		As $|\nabla\phi|$ is bounded in $\Omega'_\epsilon$ for sufficiently small $\epsilon>0$, using Lemma~\ref{lem:perimeter_phi}, we have
		\begin{align*}
			&\lim_{\epsilon\downarrow 0}\sup_{0\le v\le \epsilon}\mathcal{H}^{d-1}(\partial\Omega_v)\frac{M_\epsilon}{\eta m_\epsilon}=\frac{1}{\eta}\frac{\sup_{\partial\Omega} |\nabla \phi|}{\inf_{\partial\Omega}|\nabla \phi|}\mathcal{H}^{d-1}(\partial\Omega), \\
			&\lim_{\epsilon\downarrow 0}\sup_{0\le v\le \epsilon}\mathcal{H}^{d-1}(\partial\Omega_v)\frac{(1-\eta)^{-\frac{\kappa}{\kappa+1}}}{m_\epsilon} \epsilon^{\frac{\kappa}{\kappa+1}} = 0.
		\end{align*}
		Finally, letting $\eta\to 1$, we conclude that
		\begin{align*}
			\limsup_{t\downarrow 0}\frac{\Vol(\Omega)-Q^X_\Omega(t)}{\mu(t)}\le \frac{\sup_{\partial\Omega} |\nabla \phi|}{\inf_{\partial\Omega}|\nabla \phi|}\mathcal{H}^{d-1}(\partial\Omega)=\frac{\sup_{\partial\Omega} |\nabla \phi|}{\inf_{\partial\Omega}|\nabla \phi|}\Per(\Omega).
		\end{align*}
		This completes the proof of the theorem. 
	\end{proof}
	 The final step in the proof of Theorem~\ref{thm:main} is to construct a function $\phi$ satisfying conditions of Theorem~\ref{thm:main2} for domains with $C^{1,1}$ boundary.
	\begin{lemma}\label{lem:level_set}
		Let $\Omega\subset\R^d$ be a bounded open set with $C^{1,1}$ boundary. Then there exists a bounded $C^{1,1}$ function $\phi:\R^d\to\R$ such that 
		\begin{enumerate}
			\item $\Omega = \{x: \phi(x)>0\}$, $\R^d\setminus \overline{\Omega}=\{x:\phi(x)<0\}$, and $\partial\Omega = \{x: \phi(x)=0\}$.
			\item $|\nabla \phi(x)|=1$ in a neighborhood of $\partial\Omega$.
		\end{enumerate}
		
	\end{lemma}
	\begin{proof}
		Consider the signed distance function
		\begin{align*}
			\delta(x)=\begin{cases}
				d(x,\partial\Omega) & \mbox{if $x\in\Omega$} \\
				-d(x,\partial\Omega) & \mbox{if $x\in\Omega^c$}.
			\end{cases}
		\end{align*}
		Since $\Omega$ has $C^{1,1}$ boundary, by \cite[Theorem~8.2 (ii)]{DelfourZolesioBook} there exists $r_0>0$ such that $\delta$ is $C^{1,1}$ in $\Omega'_{r_0}:=\{x\in\R^d: |\delta(x)|<r_0\}$ and $|\nabla \delta|=1$ in $\Omega'_{r_0}$. Fix $0<r<r_0$ and define a $C^{1,1}$ function $\varphi:\R\to\R$ 
		\begin{align*}
			\varphi(x)=\begin{cases}
				x & \mbox{if $|x|<r/2$} \\
				r & \mbox{if $x\ge r$} \\
				-r & \mbox{if $x\le -r$}.
			\end{cases}
		\end{align*}
		 Consider the function $\phi=\varphi\circ \delta$. As $\delta$ is $C^{1,1}$ in $\Omega'_{r_0}$ and $\phi$ is constant on $\{x: \delta(x)\ge r\}$, $\phi$ is a bounded $C^{1,1}$ function on $\mathbb R^d$. Also, $\phi(x)>0$ if and only if $x\in\Omega$. For any $x\in\Omega'_{r/2}$, $|\nabla \phi(x)|=|\varphi'(\delta(x)) \nabla \delta(x)|=1$. This completes the proof of the lemma.
	\end{proof}
	\begin{remark} \label{rem:C11}
		The $C^{1,1}$ regularity assumption on $\partial\Omega$ in the above proof is very crucial. From \cite[Corollary~6]{NikolovPascal2025} it is known that the differentiability of $\delta$ in a neighborhood of $\partial\Omega$ implies that $\partial\Omega$ is $C^{1,1}$. Therefore, to require $\phi$ to be differentiable, $\partial\Omega$ must be $C^{1,1}$.
	\end{remark}
	\begin{proof}[Proof of Theorem~\ref{thm:main}]
	We note that $\phi$ in Lemma~\ref{lem:level_set} satisfies the conditions of Theorem~\ref{thm:main2} with $\kappa=1$, and $\Omega=\phi^{-1}(0,\infty)$. Also, $|\nabla \phi|=1$ on $\partial\Omega=\phi^{-1}(0)$. Therefore, Theorem~\ref{thm:main} follows from Theorem~\ref{thm:main2}. 
	\end{proof}
	\begin{proof}[Proof of Proposition~\ref{corr1}] Since $\P_0(\overline{X}^{(1)}_t\ge 0)=1$, for any $\epsilon>0$, using Markov inequality we obtain
		\begin{align*}
			\P_0\left(\overline{X}^{(1)}_t>\epsilon\right)\le \frac{\E_0\left[\left(\overline{X}^{(1)}_t\right)^p\right]}{\epsilon^p}.
		\end{align*}
	We also note that 
	\begin{align*}
		\E_0\left[\overline{X}^{(1)}_t\right]-\E_0\left[\overline{X}^{(1)}_t\wedge 1\right]\le \E_0\left[\overline{X}^{(1)}_t\mathbbm{1}\left\{\overline{X}^{(1)}_t>1\right\}\right].
	\end{align*}
	Therefore using the simple inequality 
	\begin{align*}
		x\mathbbm{1}\{x>1\}< x^p\mathbbm{1}\{x>1\}
	\end{align*}
	whenever $p>1$, we obtain
		\begin{align}
			0\le 1-\frac{\mu(t)}{m(t)}&\le \frac{\E_0\left[\overline{X}^{(1)}_t\mathbbm{1}\left\{\overline{X}^{(1)}_t>1\right\}\right]}{m(t)} \nonumber \\
			&\le \frac{\E_0\left[\left(\overline{X}^{(1)}_t\right)^p\mathbbm{1}\{\overline{X}^{(1)}_t>1\}\right]}{m(t)} \nonumber \\
			&\le \frac{\E_0\left[\left(\overline{X}^{(1)}_t\right)^p\right]}{m(t)} \label{eq:last_exp}
		\end{align}
		 where $m(t)=\E_0[\overline{X}^{(1)}_t]$, and the last inequality holds as $\overline{X}^{(1)}_t$ is a nonnegative random variable. The right hand side of \eqref{eq:last_exp} goes to $0$ as $t\downarrow 0$. Hence, $\mu(t)/m(t)\to 1$ as $t\downarrow 0$ and therefore, Assumption~\ref{assump1} is satisfied. The proof is concluded by applying Theorem~\ref{thm:main}.
	\end{proof}
	\section*{Acknowledgments} The author would like to thank Gennady Samorodnitsky for clarifying some technical details concerning fractional stable motions in Section~\ref{sec:fractional}. The author is also grateful to Renming Song, Tomasz Grzywny, Hyunchul Park, Masha Gordina, and Alexander Teplyaev for their helpful comments on an earlier version of the manuscript. Finally, the author is indebted to the anonymous referee for a careful reading of the manuscript and for pointing out a gap in the proof of Theorem~\ref{thm:symmetric_smooth} in an earlier version.
	\section*{Statements and declarations}
	\subsection*{Competing interest} The author has no competing interests to declare that are relevant to the content of this article.
	\subsection*{Funding} No funds, grants or other supports were received.
	\subsection*{Data availability statement} No data were used to support this study.

\providecommand{\bysame}{\leavevmode\hbox to3em{\hrulefill}\thinspace}
\providecommand{\MR}{\relax\ifhmode\unskip\space\fi MR }
\providecommand{\MRhref}[2]{%
	\href{http://www.ams.org/mathscinet-getitem?mr=#1}{#2}
}
\providecommand{\href}[2]{#2}


\begin{thebibliography}{10}
	
	\bibitem{mathoverflow}
	\emph{Continuity of hausdorff measure on level sets}, MathOverflow,
	URL:https://mathoverflow.net/q/398254 (version: 2021-08-01).
	
	\bibitem{Valverde2016}
	Luis Acu\~na Valverde, \emph{On the one dimensional spectral heat content for
		stable processes}, J. Math. Anal. Appl. \textbf{441} (2016), no.~1, 11--24.
	\MR{3488045}
	
	\bibitem{AdlerTaylorBook}
	Robert~J. Adler and Jonathan~E. Taylor, \emph{Random fields and geometry},
	Springer Monographs in Mathematics, Springer, New York, 2007. \MR{2319516}
	
	\bibitem{Aikawa_et_al2007}
	Hiroaki Aikawa, Tero Kilpel\"ainen, Nageswari Shanmugalingam, and Xiao Zhong,
	\emph{Boundary {H}arnack principle for {$p$}-harmonic functions in smooth
		{E}uclidean domains}, Potential Anal. \textbf{26} (2007), no.~3, 281--301.
	\MR{2286038}
	
	\bibitem{AmbrosioBook}
	Luigi Ambrosio, Nicola Fusco, and Diego Pallara, \emph{Functions of bounded
		variation and free discontinuity problems}, Oxford Mathematical Monographs,
	The Clarendon Press, Oxford University Press, New York, 2000. \MR{1857292}
	
	\bibitem{BertoinBook}
	Jean Bertoin, \emph{L\'evy processes}, Cambridge Tracts in Mathematics, vol.
	121, Cambridge University Press, Cambridge, 1996. \MR{1406564}
	
	\bibitem{BertoinCaballero2002}
	Jean Bertoin and Maria-Emilia Caballero, \emph{Entrance from {$0+$} for
		increasing semi-stable {M}arkov processes}, Bernoulli \textbf{8} (2002),
	no.~2, 195--205. \MR{1895890}
	
	\bibitem{BertoinYor2002}
	Jean Bertoin and Marc Yor, \emph{The entrance laws of self-similar {M}arkov
		processes and exponential functionals of {L}\'evy processes}, Potential Anal.
	\textbf{17} (2002), no.~4, 389--400. \MR{1918243}
	
	\bibitem{DelfourZolesioBook}
	M.~C. Delfour and J.-P. Zol\'esio, \emph{Shapes and geometries}, second ed.,
	Advances in Design and Control, vol.~22, Society for Industrial and Applied
	Mathematics (SIAM), Philadelphia, PA, 2011, Metrics, analysis, differential
	calculus, and optimization. \MR{2731611}
	
	\bibitem{GrzywnyParkSong2019}
	Tomasz Grzywny, Hyunchul Park, and Renming Song, \emph{Spectral heat content
		for {L}\'evy processes}, Math. Nachr. \textbf{292} (2019), no.~4, 805--825.
	\MR{3937619}
	
	\bibitem{KobayashiPark2022}
	Kei Kobayashi and Hyunchul Park, \emph{Large-time and small-time behaviors of
		the spectral heat content for time-changed stable processes}, Electron.
	Commun. Probab. \textbf{27} (2022), Paper No. 39, 11. \MR{4478125}
	
	\bibitem{KobayashiPark2023}
	\bysame, \emph{Spectral heat content for time-changed killed {B}rownian
		motions}, J. Theoret. Probab. \textbf{36} (2023), no.~2, 1148--1180.
	\MR{4591869}
	
	\bibitem{KobayashiKeiPark2024}
	\bysame, \emph{A unified approach to the small-time behavior of the spectral
		heat content for isotropic {L}\'evy processes}, Stochastic Process. Appl.
	\textbf{171} (2024), Paper No. 104331, 11. \MR{4712387}
	
	\bibitem{KobayashiKeiPark2025}
	\bysame, \emph{Heat content for {G}aussian processes: small-time asymptotic
		analysis}, Bernoulli \textbf{31} (2025), no.~2, 1606--1632. \MR{4863089}
	
	\bibitem{KonoMaejima1991}
	Norio K\^ono and Makoto Maejima, \emph{H\"older continuity of sample paths of
		some self-similar stable processes}, Tokyo J. Math. \textbf{14} (1991),
	no.~1, 93--100. \MR{1108158}
	
	\bibitem{Lamperti1972}
	John Lamperti, \emph{Semi-stable {M}arkov processes. {I}}, Z.
	Wahrscheinlichkeitstheorie und Verw. Gebiete \textbf{22} (1972), 205--225.
	\MR{307358}
	
	\bibitem{LeePark2025}
	Jaehun Lee and Hyunchul Park.
	\newblock Dichotomy in the small-time asymptotics of spectral heat content for
	{L}\'evy processes.
	\newblock {\em Electron. J. Probab.}, 31:Paper No. 142, 27, 2026.
	
	\bibitem{LewickaPeres2020}
	Marta Lewicka and Yuval Peres, \emph{Which domains have two-sided supporting
		unit spheres at every boundary point?}, Expo. Math. \textbf{38} (2020),
	no.~4, 548--558. \MR{4177956}
	
	\bibitem{LoeffenPatieSavov2019}
	R.~Loeffen, P.~Patie, and M.~Savov, \emph{Extinction time of non-{M}arkovian
		self-similar processes, persistence, annihilation of jumps and the
		{F}r\'echet distribution}, J. Stat. Phys. \textbf{175} (2019), no.~5,
	1022--1041. \MR{3959988}
	
	\bibitem{MeerschaertNaneVellaisamy2009}
	Mark~M. Meerschaert, Erkan Nane, and P.~Vellaisamy, \emph{Fractional {C}auchy
		problems on bounded domains}, Ann. Probab. \textbf{37} (2009), no.~3,
	979--1007. \MR{2537547}
	
	\bibitem{NikolovPascal2025}
	Nikolai Nikolov and Pascal~J. Thomas, \emph{Boundary regularity for the
		distance functions, and the eikonal equation}, J. Geom. Anal. \textbf{35}
	(2025), no.~8, Paper No. 230, 8. \MR{4920776}
	
	\bibitem{ParkSong2022}
	Hyunchul Park and Renming Song, \emph{Spectral heat content for
		{$\alpha$}-stable processes in {$C^{1,1}$} open sets}, Electron. J. Probab.
	\textbf{27} (2022), Paper No. 22, 19. \MR{4379201}
	
	\bibitem{PatieSrapionian2021}
	Pierre Patie and Anna Srapionyan, \emph{Self-similar {C}auchy problems and
		generalized {M}ittag-{L}effler functions}, Fract. Calc. Appl. Anal.
	\textbf{24} (2021), no.~2, 447--482. \MR{4254314}
	
	\bibitem{Pruitt1981}
	William~E. Pruitt, \emph{The growth of random walks and {L}\'evy processes},
	Ann. Probab. \textbf{9} (1981), no.~6, 948--956. \MR{632968}
	
	\bibitem{SamorodnitskyTaqquBook}
	Gennady Samorodnitsky and Murad~S. Taqqu, \emph{Stable non-{G}aussian random
		processes}, Stochastic Modeling, Chapman \& Hall, New York, 1994, Stochastic
	models with infinite variance. \MR{1280932}
	
	\bibitem{Sarkar2026}
	Rohan Sarkar, \emph{Fractional heat content asymptotics for carnot groups},
	arXiv:2601.04088 (2026).
	
	\bibitem{Sato_Book}
	Ken-iti Sato, \emph{L\'evy processes and infinitely divisible distributions},
	revised ed., Cambridge Studies in Advanced Mathematics, vol.~68, Cambridge
	University Press, Cambridge, 2013, Translated from the 1990 Japanese
	original. \MR{3185174}
	
	\bibitem{Savo1999}
	Alessandro Savo, \emph{Uniform estimates and the whole asymptotic series of the
		heat content on manifolds}, Geom. Dedicata \textbf{73} (1998), no.~2,
	181--214. \MR{1652049}
	
	\bibitem{SongVondracek2003}
	Renming Song and Zoran Vondraček, \emph{Potential theory of subordinate killed
		{B}rownian motion in a domain}, Probab. Theory Related Fields \textbf{125}
	(2003), no.~4, 578--592. \MR{1974415}
	
	\bibitem{Takashima1989}
	Keizo Takashima, \emph{Sample path properties of ergodic self-similar
		processes}, Osaka J. Math. \textbf{26} (1989), no.~1, 159--189. \MR{991287}
	
	\bibitem{vandenBerg1998}
	M.~van~den Berg, \emph{Heat content asymptotics for planar regions with cusps},
	J. London Math. Soc. (2) \textbf{57} (1998), no.~3, 677--693. \MR{1659861}
	
	\bibitem{BergDavies1989}
	M.~van~den Berg and E.~B. Davies, \emph{Heat flow out of regions in {${\bf
				R}^m$}}, Math. Z. \textbf{202} (1989), no.~4, 463--482. \MR{1022816}
	
	\bibitem{BergGilkey1994}
	M.~van~den Berg and Peter~B. Gilkey, \emph{Heat content asymptotics of a
		{R}iemannian manifold with boundary}, J. Funct. Anal. \textbf{120} (1994),
	no.~1, 48--71. \MR{1262245}
	
	\bibitem{BergLeGall1994}
	M.~van~den Berg and J.-F. Le~Gall, \emph{Mean curvature and the heat equation},
	Math. Z. \textbf{215} (1994), no.~3, 437--464. \MR{1262526}
	
	\bibitem{vandenBerg1988}
	M.~van~den Berg and S.~Srisatkunarajah, \emph{Heat equation for a region in
		{${\bf R}^2$} with a polygonal boundary}, J. London Math. Soc. (2)
	\textbf{37} (1988), no.~1, 119--127. \MR{921750}
	
\end{thebibliography}
\end{document}